\def\isarxiv{1}

\ifdefined\isarxiv
\documentclass[11pt]{article}

\usepackage{amsthm, amsmath, amssymb, graphicx, url}
\usepackage[margin=1in]{geometry}

\else
\documentclass[final,12pt]{colt2023} 

\fi

\usepackage{latexsym, amscd, amsfonts, mathrsfs, stmaryrd, tikz-cd, mathrsfs, bbm, esint, listings, moreverb, hyperref, pifont, enumitem}

\def\bE{\mathbb{E}}
\def\bP{\mathbb{P}}
\def\bR{\mathbb{R}}
\def\bZ{\mathbb{Z}}

\def\bfA{\mathbf{A}}

\def\cN{\mathcal{N}}

\def\cP{\mathcal{P}}

\def\cW{\mathcal{W}}
\def\cX{\mathcal{X}}
\def\cY{\mathcal{Y}}
\def\cZ{\mathcal{Z}}

\DeclareFontFamily{U}{mathx}{}
\DeclareFontShape{U}{mathx}{m}{n}{<-> mathx10}{}
\DeclareSymbolFont{mathx}{U}{mathx}{m}{n}
\DeclareMathAccent{\widehat}{0}{mathx}{"70}
\DeclareMathAccent{\widecheck}{0}{mathx}{"71}

\def\iidsim{\stackrel{\text{i.i.d.}}{\sim}}
\def\ol{\overline}

\def\wh{\widehat}
\def\wt{\widetilde}
\def\aas{\text{a.a.s.}}

\newcommand{\circnum}[1]{%
  \text{\ding{\the\numexpr #1+191}}%
}

\DeclareMathOperator\arctanh{\mathrm{arctanh}}

\DeclareMathOperator\Aut{\mathrm{Aut}}
\DeclareMathOperator\BEC{\mathrm{BEC}}

\DeclareMathOperator\BP{\mathrm{BP}}
\DeclareMathOperator\br{\mathrm{br}}
\DeclareMathOperator\BOHT{\mathrm{BOHT}}

\DeclareMathOperator\BSC{\mathrm{BSC}}

\DeclareMathOperator\HSBM{\mathrm{HSBM}}
\DeclareMathOperator\Id{\mathrm{Id}}
\DeclareMathOperator\KL{\mathrm{KL}}

\DeclareMathOperator\Pois{\mathrm{Pois}}

\DeclareMathOperator\SKL{\mathrm{SKL}}
\DeclareMathOperator\SNR{\mathrm{SNR}}

\DeclareMathOperator\Unif{\mathrm{Unif}}
\DeclareMathOperator\Var{\mathrm{Var}}

\ifdefined\isarxiv
\newtheorem{theorem}{Theorem}
\newtheorem{lemma}[theorem]{Lemma}
\newtheorem{proposition}[theorem]{Proposition}
\newtheorem{corollary}[theorem]{Corollary}

\theoremstyle{definition}
\newtheorem{definition}[theorem]{Definition}
\newtheorem{condition}[theorem]{Condition}

\else
\newtheorem{condition}[theorem]{Condition}
\fi

\usepackage{times}

\begin{document}

\ifdefined\isarxiv
\title{Weak Recovery Threshold for the Hypergraph Stochastic Block Model}
\date{}
\author{
Yuzhou Gu\thanks{\texttt{yuzhougu@mit.edu}. MIT.}
\and
Yury Polyanskiy\thanks{\texttt{yp@mit.edu}. MIT.}
}
\else
\title[Weak Recovery Threshold for the Hypergraph Stochastic Block Model]{Weak Recovery Threshold for the Hypergraph Stochastic Block Model}
\coltauthor{%
 \Name{Yuzhou Gu} \Email{yuzhougu@mit.edu}\\
 \addr Masaschusetts Institute of Technology
 \AND
 \Name{Yury Polyanskiy} \Email{yp@mit.edu}\\
 \addr Masaschusetts Institute of Technology%
}
\fi

\maketitle

\begin{abstract}%
  We study the weak recovery problem on the $r$-uniform hypergraph stochastic block model ($r$-HSBM) with two balanced communities. In HSBM a random graph is constructed by placing hyperedges with higher density if all vertices of a hyperedge share the same binary label, and weak recovery asks to recover a non-trivial fraction of the labels.
  We introduce a multi-terminal version of strong data processing inequalities (SDPIs), which we call the multi-terminal SDPI, and use it to prove a variety of impossibility results for weak recovery.
  In particular, we prove that weak recovery is impossible below the Kesten-Stigum (KS) threshold if $r=3,4$, or a strength parameter $\lambda$ is at least $\frac 15$.
  Prior work~\cite{pal2021community} established that weak recovery in HSBM is always possible above the KS threshold. Consequently, there is no information-computation gap for these cases, which (partially) resolves a conjecture of \cite{angelini2015spectral}. To our knowledge this is the first impossibility result for HSBM weak recovery.

  As usual, we reduce the study of non-recovery of HSBM to the study of non-reconstruction in a related broadcasting on hypertrees (BOHT) model. While we show that BOHT's reconstruction threshold coincides with KS for $r=3,4$, surprisingly, we demonstrate that for $r\ge 7$ reconstruction is possible also below KS.
  This shows an interesting phase transition in the parameter $r$, and suggests that for $r\ge 7$, there might be an information-computation gap for the HSBM. For $r=5,6$ and large degree we propose an approach for showing non-reconstruction below KS, suggesting that $r=7$ is the correct threshold for onset of the new phase.
%
\end{abstract}

\ifdefined\isarxiv
\else
\fi


\section{Introduction} \label{sec:intro}
\paragraph{Hypergraph stochastic block model.}
The stochastic block model (SBM) is a random graph model with community structures. It exhibits many interesting behaviors and has received a lot of attention in the last decade (see \cite{abbe2017community} for a survey).
The hypergraph stochastic block model (HSBM) is a generalization of SBM to hypergraphs, which arguably models real social networks better due to the existence of small clusters. It was first considered in \cite{ghoshdastidar2014consistency} and has been studied in a number of works, e.g., \cite{angelini2015spectral,ghoshdastidar2015provable,ghoshdastidar2015spectral,ghoshdastidar2017consistency,chien2018community,chien2019minimax,lin2017fundamental,ahn2018hypergraph,kim2018stochastic,cole2020exact,pal2021community,dumitriu2021partial,zhang2022exact,zhang2022sparse,dumitriu2023exact}.

We consider the $r$-uniform HSBM, where all hyperedges have the same size $r$.
The model has two parameters $a>b\in \bR_{\ge 0}$.
The HSBM hypergraph is generated as follows: Let the vertex set be $V=[n]$.
Generate a random label $X_u$ for all vertices $u\in V$ i.i.d.~$\sim \Unif(\{\pm\})$. Then, for every $S\in \binom{V}r$, if
all vertices in $S$ have the same label, add hyperedge $S$ with probability $\frac{a}{\binom n{r-1}}$;
otherwise add hyperedge $S$ with probability $\frac{b}{\binom n{r-1}}$.
We denote the model as $\HSBM(n,2,r,a,b)$ (where $2$ means there are two communities).

For SBM and HSBM the most important problem is to recover $X$ from observing only $G$.
Due to symmetry in the labels, we can only hope for recovering the communities up to a global sign flip.
Thus we define the distance between two labelings $X,Y\in \{\pm\}^V$ as
\begin{align}
  d_H(X, Y) = \min_{s\in \{\pm\}}\sum_{u\in V} \mathbbm{1}\{X_i \ne s Y_i\}.
\end{align}
There are three kinds of recovery guarantees commonly seen in the literature.
\begin{itemize}
  \item Exact recovery (strong consistency): The goal is to recover the labels exactly, i.e., to design an estimator $\wh X = \wh X(G)$ such that
  \begin{align}
    \lim_{n\to \infty} \bP[d_H(\wh X,X)=0] =1.
  \end{align}
  \item Almost exact recovery (weak consistency): The goal is the recover almost all labels, i.e., to design an estimator $\wh X = \wh X(G)$ such that
  \begin{align}
    \lim_{n\to \infty} \bP[d_H(\wh X,X)=o(n)]=1.
  \end{align}
  \item Weak recovery (partial recovery): The goal is to recover a non-trivial fraction of the labels, i.e., to design an estimator $\wh X = \wh X(G)$ such that there exists a constant $c<\frac 12$ such that
  \begin{align}
    \lim_{n\to \infty} \bP[d_H(\wh X,X) \le (c+o(1))n]=1.
  \end{align}
  Note that a trivial algorithm achieves $c = \frac 12$.
\end{itemize}
Different recovery questions are relevant in different parameter regimes.
For exact recovery and almost exact recovery, the phase transition occurs at expected degree of order $\log n$ (i.e., $a,b=\Theta(\log n)$ grows with $n$).
In this paper, we focus on the constant degree regime ($a,b$ are absolute constants), where the weak recovery problem is relevant.

The phase transition for exact recovery is known \cite{kim2018stochastic,zhang2022exact} for more general HSBMs. For weak recovery, \cite{angelini2015spectral} conjectured that a phase transition occurs at the Kesten-Stigum threshold.
The positive (algorithm) part of their conjecture has been proved by \cite{pal2021community,stephan2022sparse} in vast generality, giving an efficient weak recovery algorithm above the Kesten-Stigum threshold.
Despite the progress on the positive part, to the best of our knowledge, there are no negative (impossibility) results for any $r\ge 3$ before the current work.

For the graph (SBM) case $r=2$, the positive part was proved by \cite{massoulie2014community,mossel2018proof} and the negative part was established by \cite{mossel2015reconstruction,mossel2018proof} via reduction to the broadcasting on trees (BOT) model.
Therefore a natural idea is to study the reconstruction problem for a suitable hypergraph generalization of the BOT model, which we call the broadcasting on hypertrees (BOHT) model.
\cite{zhang2022sparse} mentioned that the difficulty in proving negative results lies in analyzing the BOHT model.
In this paper we prove impossibility of weak recovery results by proving non-reconstruction results for BOHT.

Before describing the BOHT model, we define the following useful parameters for HSBM.
\begin{itemize}
  \item For every vertex $u$, the expected number of hyperedges containing $u$ is $d\pm o(1)$, where
  \begin{align}
    d = \frac{(a-b) + 2^{r-1} b}{2^{r-1}}. \label{eqn:hsbm-d}
  \end{align}
  \item Expected number of vertices adjacent to $u$ is $\alpha\pm o(1)$, where
  \begin{align}
    \alpha= (r-1)d= (r-1) \frac{(a-b) + 2^{r-1} b}{2^{r-1}}. \label{eqn:hsbm-alpha}
  \end{align}
  \item Expected number of neighbors in the same community minus the number of neighbors in the other community is $\beta\pm o(1)$ where
  \begin{align}
    \beta = (r-1) \frac{a-b}{2^{r-1}}. \label{eqn:hsbm-beta}
  \end{align}
  \item Strength of the broadcasting channel is characterized by $\lambda\in [0, 1]$, defined as
  \begin{align}
    \lambda = \frac \beta\alpha = \frac{a-b}{a-b+2^{r-1}b}. \label{eqn:hsbm-lambda}
  \end{align}
  \item Signal-to-noise ratio (SNR), which is conjectured to govern the algorithmic weak recovery threshold for HSBM:
  \begin{align}
  \SNR:= \alpha \lambda^2 = (r-1)d\lambda^2 = \frac{(r-1) (a-b)^2}{2^{r-1} ((a-b) + 2^{r-1} b)}. \label{eqn:hsbm-snr}
  \end{align}
\end{itemize}
The Kesten-Stigum (KS) threshold is at $\SNR=1$.

\paragraph{Broadcasting on hypertrees.}
We define a general broadcasting on hypertrees (BOHT) model.
Let $q\ge 2$ (alphabet size), $r\ge 2$ (hyperedge size) be integers.
Let $\pi\in \cP([q])$ be a distribution of full support (where $\cP([q])$ denotes the space of distributions on $[q]$).
Let $B: [q] \to [q]^{r-1}$ be a probability kernel (called the broadcasting channel), satisfying
\begin{align} \label{eqn:defn-boht-prob-kernel}
  \sum_{k\in [q]} \pi_k \sum_{\substack{x\in [q]^{r-1} \\ x_i = j}} B(x_1,\ldots,x_{r-1}|k) = \pi_j \quad  \forall i\in [r-1], j\in [q].
\end{align}
Let $T$ be a (possibly random) $r$-uniform linear\footnote{Linear means that the intersection of two distinct hyperedges has size at most one.} hypertree rooted at $\rho$. The model $\BOHT(T,q,r,\pi,B)$ generates a label $\sigma_u$ for every vertex $u\in T$ via a downward process:
(1) generate $\sigma_\rho \sim \pi$ (2) given $\sigma_u$, for every downward hyperedge $S=\{u,v_1,\ldots,v_{r-1}\}$, generate $\sigma_{v_1},\ldots,\sigma_{v_{r-1}}$ according to $B(\cdot | \sigma_u)$, i.e., for every $y,x_1,\ldots,x_{r-1}\in [q]$, we have
\begin{align}
  \bP[\sigma_{v_i}=x_i\forall i\in [r-1] | \sigma_u=y] = B(x_1,\ldots,x_{r-1}|y).
\end{align}
We often consider the case where $T$ is a Galton-Watson hypertree, meaning that every vertex independently has $t\sim D$ downward hyperedges, where $D$ is a distrbution on $\bZ_{\ge 0}$.
We denote the resulting model as $\BOHT(q,r,\pi,B,D)$.
An important case is $D=\Pois(d)$, the Poisson distribution with mean $d$.
When $D$ is a singleton at $d\in \bZ_{\ge 0}$ we also denote the model as $\BOHT(q,r,\pi,B,d)$.

For $\HSBM(n,2,r,a,b)$, the corresponding BOHT model has $q=2$, $\pi=\Unif(\{\pm\})$, and $B=B_{r,\lambda}$ ($\lambda\in [0,1]$ is given by \eqref{eqn:hsbm-lambda}) where
\begin{align} \label{eqn:boht-B-special}
  B_{r,\lambda}(x_1,\ldots,x_{r-1}|y)
  = \left\{\begin{array}{ll}
  \lambda+\frac 1{2^{r-1}}(1-\lambda), & \text{if}~x_i=y\forall i\in [r-1], \\
  \frac 1{2^{r-1}}(1-\lambda), & \text{otherwise.}
  \end{array}\right.
\end{align}
We denote this model as $\BOHT(2,r,\lambda,D)$ and call it the special BOHT model.

The reconstruction problem asks whether we can gain any non-trivial information about the root given observation of far away vertices. In other words, whether the limit
\begin{align}
  \lim_{k\to \infty} I(\sigma_\rho; T_k, \sigma_{L_k})
\end{align}
is non-zero, where $L_k$ is the set of vertices at distance $k$ to the root $\rho$, and $T_k$ is the set of vertices at distance $\le k$ to $\rho$.
When the limit is non-zero, we say reconstruction is possible for the BOHT model; when the limit is zero, we say reconstruction is impossible.
It is known \cite{pal2021community} that the $r$-neighborhood (for any constant $r$) of a random vertex converges (in the sense of local weak convergence) to the Poisson hypertree described above.
Therefore non-reconstruction on a Poisson hypertree implies impossibility of weak recovery for the corresponding HSBM.

For the case $r=2$, the reconstruction threshold for the symmetric BOT model was established by \cite{bleher1995purity,evans2000broadcasting}. People have also studied generalizations of the BOT model with larger alphabet or asymmetric broadcasting channel, e.g., \cite{mossel2001reconstruction,mossel2003information,mezard2006reconstruction,borgs2006kesten,bhatnagar2010reconstruction,sly2009reconstruction,sly2011reconstruction,kulske2009symmetric,liu2019large,gu2020non,mossel2022exact}.
Nevertheless, to our knowledge, there has been no previous work studying the reconstruction problem for BOHT.

\paragraph{Belief propagation.}
The BOT and BOHT models can be studied using the belief propagation operator.
Consider the model $\BOHT(q,r,\pi,B,D)$.
Let $M_k$ denote the channel $\sigma_\rho\mapsto (T_k, \sigma_{L_k})$.
Then $(M_k)_{k\ge 0}$ satisfies the following recursion, called the belief propagation recursion:
\begin{align}
  M_{k+1} = \bE_{t\sim D}\left(M_k^{\times (r-1)} \circ B\right)^{\star t},
\end{align}
where
$(\cdot)^{\times (r-1)}$ denotes tensorization power, and $(\cdot)^{\star t}$ denotes $\star$-convolution power (see Section~\ref{sec:prelim}).
Let $\BP$ be the operator
\begin{align} \label{eqn:bp-operator}
  \BP(P) := \bE_{t\sim D}\left(P^{\times (r-1)} \circ B\right)^{\star t}
\end{align}
defined on the space of channels with input alphabet $[q]$.
Then the reconstruction problem is equivalent to asking whether the limit $\BP^\infty(\Id) := \lim_{k\to \infty} \BP^k(\Id)$ is trivial, where $\Id$ stands for the identity channel $\Id(y|x)=\mathbbm{1}\{x=y\}$.

\paragraph{Strong data processing inequalities.}
A useful tool for studying BOT models is the strong data processing inequalities (SDPIs).
They are quantitative versions of the data processing inequality (DPI), the most fundamental inequality in information theory.
The input-restricted version of SDPI states that for any Markov chain $U-X-Y$, we have
\begin{align} \label{eqn:defn:sdpi-input-res}
  I(U; Y) \le \eta_{\KL}(P_X, P_{Y|X}) I(U; X)
\end{align}
where $\eta_{\KL}(P_X, P_{Y|X})$ is a constant (called the contraction coefficient) depending only on $P_X$ and $P_{Y|X}$.
We always have $\eta_{\KL}(P_X, P_{Y|X}) \le 1$ by DPI, and the inequality is usually strict.
For any $f$-divergence, there is a corresponding version of SDPI, by replacing $I$ with $I_f$ and $\eta_{\KL}$ with another constant $\eta_f$ in \eqref{eqn:defn:sdpi-input-res}.

To apply SDPI to reconstruction problems on trees, the following equivalent form is more useful: for any Markov chain $Y-X-U$, we have
\begin{align} \label{eqn:defn:post-sdpi-input-res}
  I(U; Y) \le \eta^{(p)}_{\KL}(P_Y, P_{X|Y}) I(U; X),
\end{align}
where $\eta^{(p)}_{\KL}(P_Y, P_{X|Y})$ is a constant depending only on $P_Y$ and $P_{X|Y}$.
\eqref{eqn:defn:post-sdpi-input-res} is called the post-SDPI in~\cite{polyanskiy2023information}. Comparing \eqref{eqn:defn:sdpi-input-res} and \eqref{eqn:defn:post-sdpi-input-res} we see that $\eta_{\KL}(P_X, P_{Y|X}) = \eta^{(p)}_{\KL}(P_Y, P_{X|Y})$.

Now consider a BOT model $\BOHT(q,2,\pi,B,d)$. Note that in this case $B$ is a Markov kernel from $[q]$ to $[q]$ and $\pi B = \pi$.
Then the post-SDPI says that for any channel $P$ with input alphabet $[q]$, we have
\begin{align}
  I(\pi, P\circ B) \le \eta^{(p)}_{\KL}(\pi, B) I(\pi, P),
\end{align}
where $I(\pi, P)$ denotes the mutual information $I(X; Y)$ between two variables where $P_X=\pi$ and $P_{Y|X}=P$.
By subadditivity of mutual information under $\star$-convolution, we have
\begin{align}
  I(\pi, \BP(P)) \le d I(\pi, P\circ B) \le d \eta^{(p)}_{\KL}(\pi, B) I(\pi, P).
\end{align}
When $d \eta^{(p)}_{\KL}(\pi, B) < 1$, we have $\lim_{k\to \infty} I(\pi, \BP^k(P)) = 0$ and reconstruction is impossible.

This argument first appeared in \cite{kulske2009symmetric,formentin2009purity} with SKL information, and \cite{gu2020broadcasting} used it with mutual information to give currently best known non-reconstruction results for the Potts model in some parameter regimes. 
Although we introduced the method with $\BOHT(q,2,\pi,B,d)$ model, with slight modification it works for $\BOHT(q,2,\pi,B,D)$ or $\BOHT(T,q,2,\pi,B)$.

\paragraph{Multi-terminal SDPI.}
We generalize the above method to BOHT models with $r\ge 3$.
To this end, we introduce a multi-terminal version of the post-SDPI.
Let $\pi\in \cP([q])$ be a distribution and $B: [q]\to [q]^{r-1}$ be a probability kernel satisfying \eqref{eqn:defn-boht-prob-kernel}.
We define the multi-terminal contraction coefficient $\eta^{(m)}_{\KL}(\pi, B)$ (where $m$ stands for ``multi'') as the smallest constant such that for any channel $P$ with input alphabet $[q]$, we have
\begin{align} \label{eqn:defn:multi-sdpi}
  I(\pi, P^{\times (r-1)} \circ B) \le (r-1)\eta^{(m)}_{\KL}(\pi, B) I(\pi, P).
\end{align}
(See Figure~\ref{fig:homo-multi-sdpi} for an illustration.)
Then with a similar argument as the $r=2$ case, we can prove non-reconstruction for BOHT whenever $(r-1) d \eta^{(m,s)}_{\KL}(B) < 1$.

In the single-terminal setting, we usually distinguish pre-SDPI and post-SDPI.
In our multi-terminal setting, $B$ has one input and multiple outputs, so a multi-terminal version of post-SDPI makes more sense than that of pre-SDPI.
Therefore we call Eq.~\ref{eqn:defn:multi-sdpi} multi-terminal SDPI rather than multi-terminal post-SDPI.

For a BOHT model, if $q=2$, $\pi=\Unif(\{\pm\})$, and $B: \{\pm\}\to \{\pm\}^{r-1}$ together with the sign flip $\{\pm\}^{r-1}\to \{\pm\}^{r-1}$ is a BMS channel (see Section~\ref{sec:prelim}), then we say the model is binary symmetric.
For such models, the BP operator sends BMS channels to BMS channels.
So we could restrict $P$ to be a BMS channel and define $\eta^{(m,s)}_{\KL}(B)$ (where $s$ stands for ``symmetric'') to be the smallest constant such that \eqref{eqn:defn:multi-sdpi} holds for all BMS channels $P$.
By definition $\eta^{(m,s)}_{\KL}(B) \le \eta^{(m)}_{\KL}(\pi, B)$, so it might be able to give better non-reconstruction results than the non-BMS version.
Furthermore, due to a large number of tools dealing with BMS channels, $\eta^{(m,s)}_{\KL}(B)$ is often easier to compute than $\eta^{(m)}_{\KL}(\pi, B)$.

We could replace KL divergence in the above discussion by other $f$-divergences, and define the corresponding multi-terminal contraction coefficients. See Section~\ref{sec:multi-sdpi} for more discussions.


\paragraph{Our results.}
Our first result is non-reconstruction for BOHT based on multi-terminal SDPIs.
\begin{theorem}[Non-reconstruction for BOHT] \label{thm:boht-non-recon}
  Consider the model $\BOHT(q,r,\pi,B,D)$ where $\bE_{t\sim D} t=d$.
  \begin{enumerate}[label=(\roman*)]
    \item \label{item:thm-boht-non-recon:general}
    If
    \begin{align}
      \label{eqn:thm-boht-non-recon:general-kl} &(r-1) d \eta^{(m)}_{\KL}(\pi, B) < 1,\\
      \label{eqn:thm-boht-non-recon:general-skl} \text{or}\quad & (r-1) d \eta^{(m)}_{\SKL}(\pi, B)< 1~\text{and}~I_{\SKL}(\pi, B) < \infty,
    \end{align}
    then reconstruction is impossible.
    \item \label{item:thm-boht-non-recon:bms}
    Suppose the BOHT model is binary symmetric. If
    \begin{align}
      \label{eqn:thm-boht-non-recon:bms-kl} & (r-1) d \eta^{(m,s)}_{\KL}(B) < 1,\\
      \label{eqn:thm-boht-non-recon:bms-chi2} \text{or}\quad & (r-1) d \eta^{(m,s)}_{\chi^2}(B)< 1,\\
      \label{eqn:thm-boht-non-recon:bms-skl} \text{or}\quad & (r-1) d \eta^{(m,s)}_{\SKL}(B)< 1~\text{and}~C_{\SKL}(B) < \infty,
    \end{align}
    then reconstruction is impossible.
  \end{enumerate}
\end{theorem}
Our method can be modified to give non-reconstruction results for the $\BOHT(T,q,r,\pi,B)$ model. See Section~\ref{sec:boht-non-recon-fixed}.

We apply Theorem~\ref{thm:boht-non-recon} to the special case where $\pi=\Unif(\{\pm\})$ and $B=B_{r,\lambda}$.
We compute $\eta^{(m,s)}_f(B_{r,\lambda})$ for several cases and prove the following result.
\begin{theorem}[Non-reconstruction for special BOHT] \label{thm:boht-non-recon-special}
  Consider the special BOHT model $\BOHT(2,r,\lambda,D)$ where $\bE_{t\sim D} t=d$.
  \begin{enumerate}[label=(\roman*)]
    \item \label{item:thm-boht-non-recon-special-r34}
    For $r=3,4$, if $(r-1)d\lambda^2\le 1$, then reconstruction is impossible.
    \item \label{item:thm-boht-non-recon-special-any-d}
    For any $r\ge 5$, if
    \begin{align}
      \label{eqn:thm-boht-non-recon-special:chi2} &(r-1) d \sup_{0<\epsilon\le 1} f_{r,\lambda}(\epsilon) < 1,\\
      \label{eqn:thm-boht-non-recon-special:chi2-f} \text{where}\quad& f_{r,\lambda}(\epsilon) := \frac 1{r-1} \sum_{1\le i\le r-1} \binom{r-1}i (1-\epsilon)^{r-1-i} \epsilon^{i-1} \frac{\lambda^2}{\lambda+(1-\lambda)2^{1-i}},
    \end{align}
    then reconstruction is impossible.
    \item \label{item:thm-boht-non-recon-special-small-d}
    For any $r\ge 5$, if $\lambda\ge \frac 15$ and $(r-1)d\lambda^2\le 1$, then reconstruction is impossible.
  \end{enumerate}
\end{theorem}
Note that Theorem~\ref{thm:boht-non-recon-special}\ref{item:thm-boht-non-recon-special-small-d} is non-trivial even for $(r-1)\lambda^2 > 1$, because $\BOHT(2,r,\lambda,D)$ allows non-integer $d$ and the hypertree is infinite with positive probability whenever $(r-1)d > 1$.
The constant $\frac 15$ in Theorem~\ref{thm:boht-non-recon-special}\ref{item:thm-boht-non-recon-special-small-d} can be improved for any fixed $r$. For example, for $r=5$, the constant can be improved to $\frac 17$.

By a standard reduction, our non-reconstruction results for BOHT implies impossibility of weak recovery for the corresponding HSBM.
\begin{theorem}[Impossibility of weak recovery for HSBM] \label{thm:hsbm-special}
  Consider the model $\HSBM(n,2,r,a,b)$. Let $\BOHT(2,r,\lambda,\Pois(d))$ be the corresponding BOHT model.
  If any of the conditions in Theorem~\ref{thm:boht-non-recon-special}\ref{item:thm-boht-non-recon-special-r34}\ref{item:thm-boht-non-recon-special-any-d}\ref{item:thm-boht-non-recon-special-small-d} holds, then weak recovery is impossible.
\end{theorem}
In fact, the reduction holds for more general HSBMs. See Section~\ref{sec:hsbm}.

It is known that for any BOHT model above the Kesten-Stigum threshold, reconstruction is possible. By Theorem~\ref{thm:boht-non-recon-special}, for the special BOHT model, the KS threshold is tight for $r=3,4$ or $\lambda\ge \frac 15$.
Surprisingly, we show that for $r\ge 7$ and large $d$, reconstruction is possible below the KS threshold.
\begin{theorem}[Reconstruction for BOHT with $r\ge 7$ and large $d$] \label{thm:boht-recon-large-d}
  Consider the model $\BOHT(2,r,\lambda,d)$ or $\BOHT(2,r,\lambda,\Pois(d))$.
  For $r\ge 7$, there exists a constant $d_0=d_0(r)$ such that for all $d\ge d_0$, there exists $\lambda\in[0,1]$ such that $(r-1)d\lambda^2<1$ and reconstruction is possible.
\end{theorem}

\paragraph{Our technique.}
Our main technique for proving the non-reconstruction results is the multi-terminal SDPIs.
Theorem~\ref{thm:boht-non-recon} is a simple application of the multi-terminal SDPIs and subadditivity properties (under $\star$-convolution) of the relevant information measures.

Theorem~\ref{thm:boht-non-recon-special} is by applying Theorem~\ref{thm:boht-non-recon} and computing the relevant multi-terminal contraction coefficients, except for the critical case $(r-1)d \lambda^2=1$.
For Theorem~\ref{thm:boht-non-recon-special}\ref{item:thm-boht-non-recon-special-r34} we compute the SKL multi-terminal contraction coefficients for $r=3,4$.
For Theorem~\ref{thm:boht-non-recon-special}\ref{item:thm-boht-non-recon-special-any-d} we compute the $\chi^2$-multi-terminal contraction coefficients.
Theorem~\ref{thm:boht-non-recon-special}\ref{item:thm-boht-non-recon-special-small-d} is a corollary of \ref{item:thm-boht-non-recon-special-any-d}.
To compute the contraction coefficients, we write down an explicit description of the BP operator, and use properties of BMS channels such as BSC mixture representation and extremal BMS channels.

It turns out that for the special BOHT model, the tight reconstruction threshold can be achieved for $r=3$ using SKL or $\chi^2$-contraction, and for $r=4$ using SKL.
KL contraction does not give tight threshold for any $r\ge 3$ and $\chi^2$-contraction fails for $r\ge 4$.
In particular, there exists $\lambda\in(0,1)$ such that $\eta_{\chi^2}^{(m,s)}(B_{4,\lambda}) > \eta_{\SKL}^{(m,s)}(B_{4,\lambda}) = \lambda^2$.
This shows an important difference between single-terminal and multi-terminal SDPIs, because for single-terminal SDPIs we always have $\eta_{\chi^2}(\pi, P) \le \eta_{\SKL}(\pi, P)$ (e.g., \cite{cohen1998comparisons,raginsky2016strong,polyanskiy2017strong}).
Furthermore, before our work, to the best of our knowledge, non-reconstruction results proved via SKL information could always be also shown via other information measures ($\chi^2$-information~\cite{evans2000broadcasting}, KL information~\cite{gu2020non}, etc.). It appears, thus, that BOHT is the first example where contraction via SKL information gives better results than any other information measures we have tried.

For the critical case in Theorem~\ref{thm:boht-non-recon-special}\ref{item:thm-boht-non-recon-special-r34}\ref{item:thm-boht-non-recon-special-small-d}, some extra argument is needed. Roughly speaking, we show that the multi-terminal SDPI achieves equality only when the input channel $P$ is trivial. So any fixed point of the BP operator must be trivial.

Theorem~\ref{thm:hsbm-special} is a corollary of Theorem~\ref{thm:boht-non-recon-special} via a standard reduction which says non-reconstruction for BOHT implies impossibility of weak recovery for the corresponding HSBM.
This reduction was first proved by \cite{mossel2015reconstruction,mossel2018proof} for the two-community SBM, and we extend it to handle general HSBM.

Theorem~\ref{thm:boht-recon-large-d} is proved using Gaussian approximation at large $d$ and contraction of $\chi^2$-information.
This is a method introduced by \cite{sly2009reconstruction,sly2011reconstruction} and has proved successful in several settings \cite{liu2019large,mossel2022exact}.


\paragraph{Structure of the paper.}
In Section~\ref{sec:prelim}, we review preliminaries on information channels.
In Section~\ref{sec:multi-sdpi}, we introduce the multi-terminal SDPI and prove Theorem~\ref{thm:boht-non-recon}.
In Section~\ref{sec:non-recon-special} we study the special BOHT model $\BOHT(2,r,\lambda,D)$ and prove Theorem~\ref{thm:boht-non-recon-special}.
In Section~\ref{sec:discussion}, we discuss a possible approach to resolve the $r=5,6$ case of the special BOHT model.

In Section~\ref{sec:boht-non-recon-fixed}, we prove non-reconstruction results for BOHT models on a fixed hypertree.
In Section~\ref{sec:skl-multi-sdpi}, we compute SKL multi-terminal contraction coefficients for $B_{r,\lambda}$ with $r=3,4$.
In Section~\ref{sec:chi2-multi-sdpi}, we compute $\chi^2$-multi-terminal contraction coefficients for several binary-input symmetric channels.
In Section~\ref{sec:hsbm}, we give a general reduction from HSBM to BOHT, and prove Theorem~\ref{thm:hsbm-special}.
In Section~\ref{sec:recon-large-deg}, we prove Theorem~\ref{thm:boht-recon-large-d}, that the KS threshold is not tight for the special BOHT model with $r\ge 7$ and large $d$.



\section{Preliminaries} \label{sec:prelim}
We give necessary preliminaries on information channels, especially BMS channels. Most material in this section can be found in \cite{polyanskiy2023information} or \cite[Chapter 4]{richardson2008modern}.

\begin{definition}[BMS channels] \label{defn:bms}
  A channel $P: \{\pm\} \to \cY$ is called a binary memoryless symmetric (BMS) channel if there exists a measurable involution $\sigma: \cY\to \cY$ such that $P(E|+) = P(\sigma(E)|-)$ for all measurable subsets $E\subseteq \cY$.
\end{definition}
Binary erasure channels (BECs) and binary symmetric channels (BSCs) are the simplest examples of BMS channels. Channel $B_{r,\lambda}: \{\pm\} \to \{\pm\}^{r-1}$ defined in~\eqref{eqn:boht-B-special} (together with coordinate-wise sign flip) is also naturally a BMS channel.

BMS channels are equivalent to distributions on the interval $\left[0,\frac 12\right]$, via the following lemma.
\begin{lemma}[BSC mixture representation of BMS channels] \label{lem:bms-mixture}
  Every BMS channel $P$ is equivalent to a channel $X\to (\Delta,Z)$ where $\Delta\in \left[0,\frac 12\right]$ is independent of $X$ and $P_{Z|\Delta,X} = \BSC_\Delta(\cdot | X)$.
\end{lemma}
In the setting of Lemma~\ref{lem:bms-mixture}, we say $\Delta$ is the $\Delta$-component of $P$.
We define $\theta=1-2\Delta \in [0,1]$ to be the $\theta$-component of $P$, because it sometimes simplifies the notation.

Degradation is a very useful relationship between channels.
\begin{definition}[Degradation] \label{defn:degradation}
  Let $P: \cX \to \cY$ and $Q: \cX \to \cZ$ be two channels with the same input alphabet. We say $P$ is a degradation of $Q$, denoted $P\le_{\deg} Q$, if there exists channel $R: \cZ\to \cY$ such that $P = R\circ Q$.
\end{definition}

For any $f$-divergence, distrbution $\pi \in \cP(\cX)$ and channel $P: \cX \to \cY$, we define $I_f(\pi, P)$ as the $f$-information $I_f(X; Y)$ where $X$ is a random variable with distribution $\pi$ and $Y$ is the output of $P$ when given input $X$.
Every $f$-information respects degradation: if $P \le_{\deg} Q$, then $I_f(\pi, P) \le I_f(\pi, Q)$ for any $f$ and $\pi$.
For our purpose, the most important $f$-divergences are the KL divergence $f(x)=x\log x$, $\chi^2$-divergence $f(x)=(x-1)^2$, and symmetric KL (SKL) divergence $f(x)=(x-1)\log x$. We denote the corresponding $f$-information as $I$, $I_{\chi^2}$ and $I_{\SKL}$ respectively.

When $P$ is a BMS channel and $\pi=\Unif(\{\pm\})$, we use $C_f(P)$ to denote $I_f(\pi, P)$.
We can compute $C_f(P)$ using the $\Delta$-component. In particular, we have the following information measures.
\begin{definition}[Information measures for BMS channels]\label{defn:bms-info-measure}
  Let $P$ be a BMS channel, $\Delta$ be its $\Delta$-component, and $\theta$ be its $\theta$-component. We define the following information measures.
  \begin{align}
    C(P) & = \bE [\log 2 + \Delta \log \Delta + (1-\Delta)\log(1-\Delta)], \tag{capacity}\\
    C_{\chi^2}(P) &= \bE \theta^2, \tag{$\chi^2$-capacity}\\
    C_{\SKL}(P) &= \bE\left[\left(\frac12-\Delta\right)\log \frac{1-\Delta}{\Delta}\right] = \bE\left[\theta \arctanh\theta\right]. \tag{$\SKL$ capacity}
  \end{align}
\end{definition}

Let $P: \cX\to \cY$ and $Q: \cX'\to \cY'$ be two channels. We define the tensor product channel $P\times Q: \cX\times \cX'\to \cY\times \cY'$ by letting $P$ and $Q$ acting on the two inputs independently.
For $n\in \bZ_{\ge 1}$, we use $P^{\times n}: \cX^n \to \cY^n$ to denote the $n$-th tensor power of $P$.

Let $P: \cX\to \cY$ and $Q: \cX\to \cZ$ be two channels with the same input alphabet. We define the $\star$-convolution $P\star Q: \cX \to \cY\times \cZ$ by letting $P$ and $Q$ acting on the same input independently.
For $n\in \bZ_{\ge 0}$, we use $P^{\star n}: \cX \to \cY^n$ to denote the $n$-th $\star$-power of $P$.

Mutual information and SKL information are useful for the study of BOHT models because they are subadditive under $\star$-convolution (SKL information is even additive).
For any two channels $P,Q$ with input alphabet $\cX$ and any $\pi\in \cP(\cX)$, we have
\begin{align}
  I(\pi, P\star Q) \le I(\pi, P) + I(\pi, Q), \qquad I_{\SKL}(\pi, P\star Q) = I_{\SKL}(\pi, P) + I_{\SKL}(\pi, Q).
\end{align}
The mutual information part is standard, and the SKL information part first appeared in \cite{kulske2009symmetric}.
For $\chi^2$-information, subadditivity does not hold in general, but \cite{abbe2019subadditivity} proved that subadditivity holds when $\cX=\{\pm\}$ and $\pi = \Unif(\{\pm\})$.
That is, for any two binary-input channels $P,Q$, we have
\begin{align}
  I_{\chi^2}(\Unif(\{\pm\}), P\star Q) \le I_{\chi^2}(\Unif(\{\pm\}), P) + I_{\chi^2}(\Unif(\{\pm\}), Q).
\end{align}


\section{Multi-terminal SDPI} \label{sec:multi-sdpi}
Let $q\ge 2$, $r\ge 2$ be integers, $\pi\in \cP([q])$, and $B: [q]\to [q]^{r-1}$ be a channel satisfying \eqref{eqn:defn-boht-prob-kernel}.
For any $f$-divergence, we define the multi-terminal contraction coefficient as
\begin{align} \label{eqn:defn-homo-multi-eta}
  \eta^{(m)}_f(\pi, B) := \sup_{P} \frac{I_f(\pi, P^{\times (r-1)}\circ B)}{(r-1)I_f(\pi,P)}
\end{align}
where $P$ goes over all channels with input alphabet $[q]$ for which $0 < I_f(\pi, P) < \infty$.

In other words, $\eta^{(m)}_f(\pi, B)$ is the smallest constant such that for any diagram as in Figure~\ref{fig:homo-multi-sdpi} where $P_X=\pi$, $P_{Y^{r-1}|X} = B$, $P_{U_i|Y_i}=P$, we have
\begin{align}
  I_f(X; U^{r-1}) \le (r-1) \eta^{(m)}_f(\pi, B) I_f(Y_1; U_1).
\end{align}
\begin{figure}[ht]
	\centering
	\includegraphics{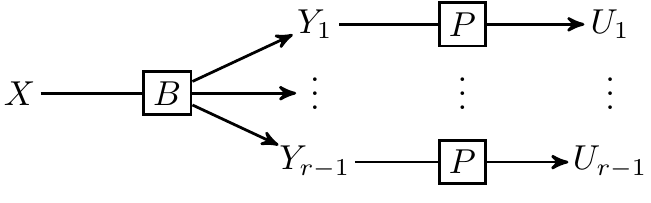}
  \caption{Setting for homogeneous multi-terminal SDPI}
  \label{fig:homo-multi-sdpi}
\end{figure}

In the above definition, we assume $P_{U_i|Y_i}$ are the same for all $i\in [r-1]$. Therefore \eqref{eqn:defn-homo-multi-eta} defines a homogeneous multi-terminal contraction coefficient.
It is also possible to define a heterogeneous version, where the input channels can be different.
We define
\begin{align} \label{eqn:defn-hetero-multi-eta}
  \eta^{(m,ht)}_f(\pi, B) := \sup_{P_1,\ldots,P_{r-1}} \frac{I_f(\pi, (P_1\times \cdots \times P_{r-1})\circ B)}{\sum_{i\in [r-1]} I_f(\pi, P_i)}
\end{align}
where $P_1,\ldots,P_{r-1}$ goes over all channels with input alphabet $[q]$ for which $0 < \sum_{i\in [r-1]} I_f(\pi, P_i) < \infty$ (here $ht$ stands for ``heterogeneous'').
In other words, $\eta^{(m,ht)}_f(\pi, B)$ is the smallest constant such that for any diagram as in Figure~\ref{fig:hetero-multi-sdpi} where $P_X = \pi$, we have
\begin{align}
  I_f(X; U^{r-1}) \le \eta^{(m,ht)}_f(\pi, B) \sum_{i\in [r-1]} I_f(Y_i; U_i).
\end{align}
\begin{figure}[ht]
  \centering
	\includegraphics{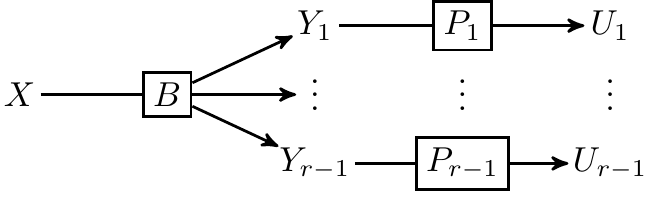}
  \caption{Setting for heterogeneous multi-terminal SDPI}
  \label{fig:hetero-multi-sdpi}
\end{figure}

It is clear from definition that
\begin{align}
  \eta^{(m)}_f(\pi, B) \le \eta^{(m,ht)}_f(\pi, B).
\end{align}

Unlike the usual contraction coefficients, it is not true in general that $\eta^{(m)}_f(\pi, B) \le 1$. Nevertheless, this holds when the $f$-information is subadditive under $\star$-convolution.
For the mutual information, we have
\begin{align} \label{eqn:kl-multi-eta-le-one}
  I(X; U^{r-1}) \le I(Y^{r-1}; U^{r-1}) \le \sum_{i\in [r-1]} I(Y^{r-1}; U_i) = \sum_{i\in [r-1]} I(Y_i, U_i),
\end{align}
where the first step is by DPI and the second step is by subadditivity.
Therefore
\begin{align}
  \eta^{(m)}_{\KL}(\pi, B) \le \eta^{(m,ht)}_{\KL}(\pi, B) \le 1.
\end{align}
The same holds for the SKL mutual information, so
\begin{align}
  \eta^{(m)}_{\SKL}(\pi, B) \le \eta^{(m,ht)}_{\SKL}(\pi, B) \le 1.
\end{align}

When $q=2$, $\pi=\Unif(\{\pm\})$, and $B$ together with the sign flip is a BMS channel, for any BMS channels $P_1,\ldots,P_{r-1}$, the combined channel $(P_1\times \cdots \times P_{r-1}) \circ B$ is also a BMS channel.
In this case we can define versions of multi-terminal contraction coefficients restricted to BMS channels.
We define
\begin{align}
  \eta^{(m,s)}_f(B) := \sup_{P} \frac{C_f(P^{\times (r-1)}\circ B)}{(r-1) C_f(P)}
\end{align}
where $P$ goes over BMS channels with $0<C_f(P) < \infty$, and
\begin{align}
  \eta^{(m,ht,s)}_f(B) := \sup_{P_1,\ldots,P_{r-1}} \frac{C_f((P_1\times \cdots \times P_{r-1})\circ B)}{\sum_{i\in [r-1]} C_f(P_i)}
\end{align}
where $P_1,\ldots,P_{r-1}$ goes over BMS channels with $0 < \sum_{i\in [r-1]} C_f(P_{i}) < \infty$.
Because $\chi^2$-capacity is subadditive over BMS channels, by a similar computation as \eqref{eqn:kl-multi-eta-le-one} we have
\begin{align}
  \eta^{(m,s)}_{\chi^2}(B) \le \eta^{(m,ht,s)}_{\chi^2}(B) \le 1.
\end{align}

With these definitions, it is very easy to prove Theorem~\ref{thm:boht-non-recon}.
\begin{proof}[Proof of Theorem~\ref{thm:boht-non-recon}]
  Let $M_k$ denote the channel $\sigma_\rho \mapsto (T_k, \sigma_{L_k})$. Then $(M_k)_{k\ge 0}$ satisfies the BP recursion $M_{k+1}=\BP(M_k)$, where $\BP$ is defined in \eqref{eqn:bp-operator}.

  \paragraph{Part \ref{item:thm-boht-non-recon:general}, mutual information:}
  For any channel $P$ with input alphabet $[q]$ we have
  \begin{align}
    I(\pi, \BP(P)) \le d I(\pi, P^{\times (r-1)} \circ B) \le (r-1) d \eta^{(m)}_{\KL}(\pi, B) I(\pi, P),
  \end{align}
  where the first step is by subadditivity of mutual information, and the second step is by definition of multi-terminal contraction coefficients.
  If \eqref{eqn:thm-boht-non-recon:general-kl} holds, then $I(\pi, M_{k+1}) \le c I(\pi, M_k)$ for $c=(r-1) d \eta^{(m)}_{\KL}(\pi, B)<1$.
  Because $I(\pi, M_0)<\infty$, we have $\lim_{k\to \infty} I(\pi, M_k)=0$ and non-reconstruction holds.

  \paragraph{Part \ref{item:thm-boht-non-recon:general}, SKL information:}
  If \eqref{eqn:thm-boht-non-recon:general-skl} holds, then
  \begin{align}
    I_{\SKL}(\pi, M_1) = d I_{\SKL}(\pi, P^{\times (r-1)} \circ B)
    \le d I_{\SKL}(\pi, B) < \infty,
  \end{align}
  where the first step is by additivity of SKL information, and the second step is by DPI.
  The rest of the proof is similar to the previous case.

  \paragraph{Part \ref{item:thm-boht-non-recon:bms}, KL and SKL capacity:}
  The channels $M_k$ are all BMS channels, so we can use the BMS version of multi-terminal contraction coefficients. The rest of the proof is the same as Part \ref{item:thm-boht-non-recon:general}.

  \paragraph{Part \ref{item:thm-boht-non-recon:bms}, $\chi^2$-capacity:}
  Use subadditivity of $\chi^2$-capacity for BMS channels and the proof is similar to previous cases.
\end{proof}
See Section~\ref{sec:boht-non-recon-fixed} for a variation of Theorem~\ref{thm:boht-non-recon} to the fixed-hypertree model $\BOHT(T,q,r,\pi,B)$.

\section{Non-reconstruction for the special BOHT model} \label{sec:non-recon-special}
In this section we focus on the special BOHT model $\BOHT(2,r,\lambda,D)$, which is the BOHT model with $q=2$, $\pi=\Unif(\{\pm\})$, and $B=B_{r,\lambda}$ as defined in \eqref{eqn:boht-B-special}.
We prove non-reconstruction results for this model using Theorem~\ref{thm:boht-non-recon}\ref{item:thm-boht-non-recon-fixed:bms} by computing the relevant multi-terminal contraction coefficients.



To compute the contraction coefficients, we need to descibe $P^{\times (r-1)} \circ B_{r,\lambda}$ (where $P$ is a BMS channel) in a more explicit way. By BSC mixture representation, we only need to describe $(\BSC_{\Delta_1}\times \cdots \times \BSC_{\Delta_{r-1}})\circ B_{r,\lambda}$.
Let $\theta_i := 1-2\Delta_i$ for $i\in [r-1]$.
For $x\in \{\pm\}^{r-1}$, we have
\begin{align}
  & ((\BSC_{\Delta_1}\times \cdots \times \BSC_{\Delta_{r-1}})\circ B_{r,\lambda})(x_1,\ldots,x_{r-1} | +) \\
  \nonumber =&~ \sum_{y\in \{\pm\}^{r-1}} B_{r,\lambda}(y_1,\ldots,y_{r-1} |+) \prod_{i\in [r-1]} \BSC_{\Delta_i}(x_i | y_i) \\
  \nonumber =&~ \lambda \prod_{i\in [r-1]} \BSC_{\Delta_i}(x_i | +)
  + \frac 1{2^{r-1}} (1-\lambda) \prod_{i\in [r-1]} \sum_{y_i\in\{\pm\}} \BSC_{\Delta_i}(x_i| y_i) \\
  \nonumber =&~ \lambda \prod_{i\in [r-1]} \left(\frac12 + \left(\frac12-\Delta_i\right) x_i\right) + \frac 1{2^{r-1}}(1-\lambda)\\
  \nonumber =&~ \lambda \prod_{i\in [r-1]} \left(\frac12 + \frac12 \theta_i x_i\right) + \frac 1{2^{r-1}}(1-\lambda).
\end{align}

So $(\BSC_{\Delta_1}\times \cdots \times \BSC_{\Delta_{r-1}})\circ B_{r,\lambda}$ is a mixture of $2^{r-2}$ BSCs,
indexed by the set
\begin{align}
  \left\{x: x\in \{\pm\}^{r-1}, x_1=+\right\},
\end{align}
where the BSC corresponding to $x$ has weight (probability)
\begin{align}
  & \left(\lambda \prod_{i\in [r-1]} \left(\frac12 + \frac12\theta_i x_i\right) + \frac 1{2^{r-1}}(1-\lambda)\right)
   + \left(\lambda \prod_{i\in [r-1]} \left(\frac12 - \frac12\theta_i x_i\right) + \frac 1{2^{r-1}}(1-\lambda)\right)\\
  \nonumber =&~ \lambda \left(\prod_{i\in [r-1]} \left(\frac12 + \frac12\theta_i x_i\right) +\prod_{i\in [r-1]} \left(\frac12 - \frac12\theta_i x_i\right)\right) + \frac 1{2^{r-2}}(1-\lambda)
\end{align}
and $\theta$ parameter equal to the absolute value of
\begin{align}
  &~\frac{\left(\lambda \prod_{i\in [r-1]} \left(\frac12 + \frac12\theta_i x_i\right) + \frac 1{2^{r-1}}(1-\lambda)\right)
  - \left(\lambda \prod_{i\in [r-1]} \left(\frac12 - \frac12\theta_i x_i\right) + \frac 1{2^{r-1}}(1-\lambda)\right)}{\left(\lambda \prod_{i\in [r-1]} \left(\frac12 + \frac12\theta_i x_i\right) + \frac 1{2^{r-1}}(1-\lambda)\right)
  + \left(\lambda \prod_{i\in [r-1]} \left(\frac12 - \frac12\theta_i x_i\right) + \frac 1{2^{r-1}}(1-\lambda)\right)}\\
  \nonumber =&~\frac{\lambda \left( \prod_{i\in [r-1]} \left(\frac12 + \frac12\theta_i x_i\right) - \prod_{i\in [r-1]} \left(\frac12 - \frac12\theta_i x_i\right)\right)}{\lambda \left(\prod_{i\in [r-1]} \left(\frac12 + \frac12\theta_i x_i\right) +\prod_{i\in [r-1]} \left(\frac12 - \frac12\theta_i x_i\right)\right) + \frac 1{2^{r-2}}(1-\lambda)}\\
  \nonumber =&~ \frac{\lambda \left( \prod_{i\in [r-1]} \left(1+\theta_i x_i\right) - \prod_{i\in [r-1]} \left(1-\theta_i x_i\right)\right)}{\lambda \left(\prod_{i\in [r-1]} \left(1+\theta_i x_i\right) +\prod_{i\in [r-1]} \left(1-\theta_i x_i\right)\right) + 2(1-\lambda)}.
\end{align}
Although we need to take absolute value, information measures (except for $P_e$) in Definition~\ref{defn:bms-info-measure} are all even functions in $\theta$.
So we do not need to worry about the sign.

Using this explicit description of $P^{\times (r-1)} \circ B_{r,\lambda}$, we are able to compute several multi-terminal contraction coefficients of $B_{r,\lambda}$.

For Theorem~\ref{thm:boht-non-recon-special}\ref{item:thm-boht-non-recon-special-r34}, we compute the SKL contraction coefficients.
\begin{proposition}[SKL contraction coefficient]
  \label{prop:skl-multi-sdpi}
  Fix $r=3$ or $4$ and $\lambda\in [0,1]$. Then
  \begin{align}
    \eta^{(m,ht,s)}_{\SKL}(B_{r,\lambda}) = \lambda^2.
  \end{align}
  Furthermore, if $0<\lambda<1$ and $P_1,\ldots,P_{r-1}$ are BMS channels with at least one $P_i$ non-trivial, then
  \begin{align} \label{eqn:prop-skl-multi-sdpi-strict}
    C_{\SKL}((P_1\times \cdots \times P_{r-1})\circ B_{r,\lambda}) < \lambda^2 \sum_{i\in [r-1]} C_{\SKL}(P_i).
  \end{align}
\end{proposition}
Proof of Prop.~\ref{prop:skl-multi-sdpi} is deferred to Section~\ref{sec:skl-multi-sdpi}.
\begin{proof}[Proof of Theorem~\ref{thm:boht-non-recon-special}\ref{item:thm-boht-non-recon-special-r34}]
  If $(r-1)d\le 1$, then the BOHT model extincts almost surely (e.g., \cite[Prop.~5.4]{lyons2017probability}). This resolves the case $\lambda=1$. In the following, assume that $0\le \lambda<1$.

  For $0\le \lambda<1$, we have $C_{\SKL}(B_{r,\lambda}) < \infty$. Therefore Theorem~\ref{thm:boht-non-recon}\ref{item:thm-boht-non-recon:bms} together with Prop.~\ref{prop:skl-multi-sdpi} implies that non-reconstruction holds for $(r-1)d\lambda^2 < 1$.

  For the critical case $(r-1)d\lambda^2 = 1$ we need some extra argument. See Section~\ref{sec:skl-multi-sdpi:critical}.
\end{proof}

For Theorem~\ref{thm:boht-non-recon-special}\ref{item:thm-boht-non-recon-special-any-d}, we compute the $\chi^2$-contraction coefficients.
\begin{proposition}[$\chi^2$-contraction coefficient]
  \label{prop:chi2-multi-sdpi}
  Fix $r\in \bZ_{\ge 3}$. Then
  \begin{align} \label{eqn:prop-chi2-multi-sdpi}
    \eta^{(m,s)}_{\chi^2}(B_{r,\lambda}) = \sup_{0<\epsilon\le 1} f_{r,\lambda}(\epsilon),
  \end{align}
  where $f_{r,\lambda}$ is defined in \eqref{eqn:thm-boht-non-recon-special:chi2-f}.
\end{proposition}
Proof of Prop.~\ref{prop:chi2-multi-sdpi} is deferred to Section~\ref{sec:chi2-multi-sdpi}.
\begin{proof}[Proof of Theorem~\ref{thm:boht-non-recon-special}\ref{item:thm-boht-non-recon-special-any-d}]
  Follows from Theorem~\ref{thm:boht-non-recon}\ref{item:thm-boht-non-recon-fixed:bms} and Prop.~\ref{prop:chi2-multi-sdpi}.
\end{proof}

Interestingly, RHS of \eqref{eqn:prop-chi2-multi-sdpi} can be computed exactly when $\lambda$ is not too small.
\begin{lemma} \label{lem:chi2-multi-sdpi-large-lambda}
  Fix $r\in \bZ_{\ge 3}$ and $\lambda\in \left[\frac 15, 1\right]$.
  For all $0<\epsilon\le 1$, we have $f_{r,\lambda}(\epsilon) < \lambda^2$.
\end{lemma}
Proof of Lemma~\ref{lem:chi2-multi-sdpi-large-lambda} is deferred to Section~\ref{sec:chi2-multi-sdpi}.
\begin{proof}[Proof of Theorem~\ref{thm:boht-non-recon-special}\ref{item:thm-boht-non-recon-special-small-d}]
  Prop.~\ref{prop:chi2-multi-sdpi} together with Lemma~\ref{lem:chi2-multi-sdpi-large-lambda} implies that for $\lambda\in \left[\frac 15, 1\right]$ we have
  \begin{align}
    \eta^{(m,s)}_{\chi^2}(B_{r,\lambda}) = \lambda^2.
  \end{align}
  (For the lower bound, take $\epsilon\to 0^+$ in \eqref{eqn:prop-chi2-multi-sdpi}.)
  Then Theorem~\ref{thm:boht-non-recon}\ref{item:thm-boht-non-recon-fixed:bms} implies that non-reconstruction holds for $(r-1)d\lambda^2<1$.

  For the critical case $(r-1)d\lambda^2 = 1$ we need some extra argument. See Section~\ref{sec:chi2-multi-sdpi:critical}.
\end{proof}

\section{Discussions} \label{sec:discussion}
For the special BOHT model, we have left the $r=5,6$ case open.
Our preliminary computations suggest that for $r=5,6$, there exists an absolute constant $d_0\in \bR_{\ge 0}$ such that the BOHT model has non-reconstruction when $d\ge d_0$ and $(r-1)d\lambda^2\le 1$.
We believe that a generalization of Sly's method~\cite{sly2011reconstruction,mossel2022exact} can be used to prove this. In Sly's method, we compute the first few orders of the BP recursion formula. Combined with Gaussian approximation this would imply contraction of $\chi^2$-capacity.
One technical challenge is that in the BOHT case we need a two-step application of Sly's method, in contrast with previous works.

\ifdefined\isarxiv
\section*{Acknowledgments}
We thank Kunal Marwaha for resolving a conjecture in a preliminary version of the paper. We thank the anonymous reviewers for helpful comments.

Research was sponsored by the United States Air Force Research Laboratory and the United States Air Force Artificial Intelligence Accelerator and was accomplished under Cooperative Agreement Number FA8750-19-2-1000. The views and conclusions contained in this document are those of the authors and should not be interpreted as representing the official policies, either expressed or implied, of the United States Air Force or the U.S. Government. The U.S. Government is authorized to reproduce and distribute reprints for Government purposes notwithstanding any copyright notation herein.
\else
\acks{We thank Kunal Marwaha for resolving a conjecture in a preliminary version of the paper. We thank the anonymous reviewers for helpful comments.

Research was sponsored by the United States Air Force Research Laboratory and the United States Air Force Artificial Intelligence Accelerator and was accomplished under Cooperative Agreement Number FA8750-19-2-1000. The views and conclusions contained in this document are those of the authors and should not be interpreted as representing the official policies, either expressed or implied, of the United States Air Force or the U.S. Government. The U.S. Government is authorized to reproduce and distribute reprints for Government purposes notwithstanding any copyright notation herein.}
\fi

\ifdefined\isarxiv
\bibliographystyle{alpha}
\fi
\bibliography{ref}

\begin{thebibliography}{BCMR06}

\bibitem[ABA19]{abbe2019subadditivity}
Emmanuel Abbe and Enric Boix-Adser{\`a}.
\newblock Subadditivity beyond trees and the chi-squared mutual information.
\newblock In {\em 2019 IEEE International Symposium on Information Theory
  (ISIT)}, pages 697--701. IEEE, 2019.

\bibitem[Abb17]{abbe2017community}
Emmanuel Abbe.
\newblock Community detection and stochastic block models: recent developments.
\newblock {\em The Journal of Machine Learning Research}, 18(1):6446--6531,
  2017.

\bibitem[ACKZ15]{angelini2015spectral}
Maria~Chiara Angelini, Francesco Caltagirone, Florent Krzakala, and Lenka
  Zdeborov{\'a}.
\newblock Spectral detection on sparse hypergraphs.
\newblock In {\em 2015 53rd Annual Allerton Conference on Communication,
  Control, and Computing (Allerton)}, pages 66--73. IEEE, 2015.

\bibitem[ALS18]{ahn2018hypergraph}
Kwangjun Ahn, Kangwook Lee, and Changho Suh.
\newblock Hypergraph spectral clustering in the weighted stochastic block
  model.
\newblock {\em IEEE Journal of Selected Topics in Signal Processing},
  12(5):959--974, 2018.

\bibitem[BCMR06]{borgs2006kesten}
Christian Borgs, Jennifer Chayes, Elchanan Mossel, and S{\'e}bastien Roch.
\newblock The kesten-stigum reconstruction bound is tight for roughly symmetric
  binary channels.
\newblock In {\em 2006 47th Annual IEEE Symposium on Foundations of Computer
  Science (FOCS'06)}, pages 518--530. IEEE, 2006.

\bibitem[BLM15]{bordenave2015non}
Charles Bordenave, Marc Lelarge, and Laurent Massouli{\'e}.
\newblock Non-backtracking spectrum of random graphs: community detection and
  non-regular ramanujan graphs.
\newblock In {\em 2015 IEEE 56th Annual Symposium on Foundations of Computer
  Science}, pages 1347--1357. IEEE, 2015.

\bibitem[BRZ95]{bleher1995purity}
Pavel~M. Bleher, Jean Ruiz, and Valentin~A. Zagrebnov.
\newblock On the purity of the limiting gibbs state for the ising model on the
  bethe lattice.
\newblock {\em Journal of Statistical Physics}, 79:473--482, 1995.

\bibitem[BST10]{bhatnagar2010reconstruction}
Nayantara Bhatnagar, Allan Sly, and Prasad Tetali.
\newblock Reconstruction threshold for the hardcore model.
\newblock In {\em Approximation, Randomization, and Combinatorial Optimization.
  Algorithms and Techniques: 13th International Workshop, APPROX 2010, and 14th
  International Workshop, RANDOM 2010, Barcelona, Spain, September 1-3, 2010.
  Proceedings}, pages 434--447. Springer, 2010.

\bibitem[CKZ98]{cohen1998comparisons}
Joel Cohen, Johannes~HB Kempermann, and Gheorghe Zbaganu.
\newblock {\em Comparisons of stochastic matrices with applications in
  information theory, statistics, economics and population}.
\newblock Springer Science \& Business Media, 1998.

\bibitem[CLM17]{caltagirone2017recovering}
Francesco Caltagirone, Marc Lelarge, and L{\'e}o Miolane.
\newblock Recovering asymmetric communities in the stochastic block model.
\newblock {\em IEEE Transactions on Network Science and Engineering},
  5(3):237--246, 2017.

\bibitem[CLW18]{chien2018community}
I~Chien, Chung-Yi Lin, and I-Hsiang Wang.
\newblock Community detection in hypergraphs: Optimal statistical limit and
  efficient algorithms.
\newblock In {\em International Conference on Artificial Intelligence and
  Statistics}, pages 871--879. PMLR, 2018.

\bibitem[CLW19]{chien2019minimax}
I~Eli Chien, Chung-Yi Lin, and I-Hsiang Wang.
\newblock On the minimax misclassification ratio of hypergraph community
  detection.
\newblock {\em IEEE Transactions on Information Theory}, 65(12):8095--8118,
  2019.

\bibitem[CS20]{chin2020optimal}
Byron Chin and Allan Sly.
\newblock Optimal recovery of block models with $ q $ communities.
\newblock {\em arXiv preprint arXiv:2010.10672}, 2020.

\bibitem[CS21]{chin2021optimal}
Byron Chin and Allan Sly.
\newblock Optimal reconstruction of general sparse stochastic block models.
\newblock {\em arXiv preprint arXiv:2111.00697}, 2021.

\bibitem[CZ20]{cole2020exact}
Sam Cole and Yizhe Zhu.
\newblock Exact recovery in the hypergraph stochastic block model: A spectral
  algorithm.
\newblock {\em Linear Algebra and its Applications}, 593:45--73, 2020.

\bibitem[DW23]{dumitriu2023exact}
Ioana Dumitriu and Haixiao Wang.
\newblock Exact recovery for the non-uniform hypergraph stochastic block model.
\newblock {\em arXiv preprint arXiv:2304.13139}, 2023.

\bibitem[DWZ21]{dumitriu2021partial}
Ioana Dumitriu, Haixiao Wang, and Yizhe Zhu.
\newblock Partial recovery and weak consistency in the non-uniform hypergraph
  stochastic block model.
\newblock {\em arXiv preprint arXiv:2112.11671}, 2021.

\bibitem[EKPS00]{evans2000broadcasting}
William Evans, Claire Kenyon, Yuval Peres, and Leonard~J. Schulman.
\newblock Broadcasting on trees and the ising model.
\newblock {\em Annals of Applied Probability}, pages 410--433, 2000.

\bibitem[FK09]{formentin2009purity}
Marco Formentin and Christof K{\"u}lske.
\newblock On the purity of the free boundary condition potts measure on random
  trees.
\newblock {\em Stochastic Processes and their Applications}, 119(9):2992--3005,
  2009.

\bibitem[GD14]{ghoshdastidar2014consistency}
Debarghya Ghoshdastidar and Ambedkar Dukkipati.
\newblock Consistency of spectral partitioning of uniform hypergraphs under
  planted partition model.
\newblock {\em Advances in Neural Information Processing Systems}, 27, 2014.

\bibitem[GD15a]{ghoshdastidar2015provable}
Debarghya Ghoshdastidar and Ambedkar Dukkipati.
\newblock A provable generalized tensor spectral method for uniform hypergraph
  partitioning.
\newblock In {\em International Conference on Machine Learning}, pages
  400--409. PMLR, 2015.

\bibitem[GD15b]{ghoshdastidar2015spectral}
Debarghya Ghoshdastidar and Ambedkar Dukkipati.
\newblock Spectral clustering using multilinear {SVD}: Analysis, approximations
  and applications.
\newblock In {\em Proceedings of the AAAI Conference on Artificial
  Intelligence}, volume~29, 2015.

\bibitem[GD17]{ghoshdastidar2017consistency}
Debarghya Ghoshdastidar and Ambedkar Dukkipati.
\newblock Consistency of spectral hypergraph partitioning under planted
  partition model.
\newblock {\em The Annals of Statistics}, 45(1):289--315, 2017.

\bibitem[GP20]{gu2020non}
Yuzhou Gu and Yury Polyanskiy.
\newblock Non-linear log-{Sobolev} inequalities for the {Potts} semigroup and
  applications to reconstruction problems.
\newblock {\em arXiv preprint arXiv:2005.05444}, 2020.

\bibitem[GRP20]{gu2020broadcasting}
Yuzhou Gu, Hajir Roozbehani, and Yury Polyanskiy.
\newblock Broadcasting on trees near criticality.
\newblock In {\em 2020 IEEE International Symposium on Information Theory
  (ISIT)}, pages 1504--1509. IEEE, 2020.

\bibitem[Gu23]{gu2023channel}
Yuzhou Gu.
\newblock {\em Channel Comparison Methods and Statistical Problems on Graphs}.
\newblock PhD thesis, Massachusetts Institute of Technology, 2023.

\bibitem[KBG18]{kim2018stochastic}
Chiheon Kim, Afonso~S. Bandeira, and Michel~X. Goemans.
\newblock Stochastic block model for hypergraphs: Statistical limits and a
  semidefinite programming approach.
\newblock {\em arXiv preprint arXiv:1807.02884}, 2018.

\bibitem[KF09]{kulske2009symmetric}
Christof K{\"u}lske and Marco Formentin.
\newblock A symmetric entropy bound on the non-reconstruction regime of
  {Markov} chains on {Galton-Watson} trees.
\newblock {\em Electronic Communications in Probability}, 14:587--596, 2009.

\bibitem[KM77]{korner1977comparison}
J.~K\"orner and K.~Marton.
\newblock Comparison of two noisy channels.
\newblock {\em Topics in Information Theory, I.~Csisz\'ar and P.~Elias, Eds.,
  Amsterdam, The Netherlands}, pages 411--423, 1977.

\bibitem[LCW17]{lin2017fundamental}
Chung-Yi Lin, I~Eli Chien, and I-Hsiang Wang.
\newblock On the fundamental statistical limit of community detection in random
  hypergraphs.
\newblock In {\em 2017 IEEE International Symposium on Information Theory
  (ISIT)}, pages 2178--2182. IEEE, 2017.

\bibitem[LN19]{liu2019large}
Wenjian Liu and Ning Ning.
\newblock Large degree asymptotics and the reconstruction threshold of the
  asymmetric binary channels.
\newblock {\em Journal of Statistical Physics}, 174:1161--1188, 2019.

\bibitem[LP17]{lyons2017probability}
Russell Lyons and Yuval Peres.
\newblock {\em Probability on trees and networks}, volume~42.
\newblock Cambridge University Press, 2017.

\bibitem[Lyo90]{lyons1990random}
Russell Lyons.
\newblock Random walks and percolation on trees.
\newblock {\em The Annals of Probability}, 18(3):931--958, 1990.

\bibitem[Mar23]{marwaha2023useful}
Kunal Marwaha.
\newblock A useful inequality of inverse hyperbolic tangent.
\newblock {\em arXiv preprint arXiv:2305.18348}, 2023.

\bibitem[Mas14]{massoulie2014community}
Laurent Massouli{\'e}.
\newblock Community detection thresholds and the weak ramanujan property.
\newblock In {\em Proceedings of the forty-sixth annual ACM Symposium on Theory
  of Computing}, pages 694--703, 2014.

\bibitem[MM06]{mezard2006reconstruction}
Marc M{\'e}zard and Andrea Montanari.
\newblock Reconstruction on trees and spin glass transition.
\newblock {\em Journal of statistical physics}, 124:1317--1350, 2006.

\bibitem[MNS15]{mossel2015reconstruction}
Elchanan Mossel, Joe Neeman, and Allan Sly.
\newblock Reconstruction and estimation in the planted partition model.
\newblock {\em Probability Theory and Related Fields}, 162:431--461, 2015.

\bibitem[MNS18]{mossel2018proof}
Elchanan Mossel, Joe Neeman, and Allan Sly.
\newblock A proof of the block model threshold conjecture.
\newblock {\em Combinatorica}, 38(3):665--708, 2018.

\bibitem[Mos01]{mossel2001reconstruction}
Elchanan Mossel.
\newblock Reconstruction on trees: beating the second eigenvalue.
\newblock {\em The Annals of Applied Probability}, 11(1):285--300, 2001.

\bibitem[MP03]{mossel2003information}
Elchanan Mossel and Yuval Peres.
\newblock Information flow on trees.
\newblock {\em The Annals of Applied Probability}, 13(3):817--844, 2003.

\bibitem[MP18]{makur2018comparison}
Anuran Makur and Yury Polyanskiy.
\newblock Comparison of channels: Criteria for domination by a symmetric
  channel.
\newblock {\em IEEE Transactions on Information Theory}, 64(8):5704--5725,
  2018.

\bibitem[MSS22]{mossel2022exact}
Elchanan Mossel, Allan Sly, and Youngtak Sohn.
\newblock Exact phase transitions for stochastic block models and
  reconstruction on trees.
\newblock {\em arXiv preprint arXiv:2212.03362}, 2022.

\bibitem[PW17]{polyanskiy2017strong}
Yury Polyanskiy and Yihong Wu.
\newblock Strong data-processing inequalities for channels and bayesian
  networks.
\newblock In {\em Convexity and Concentration}, pages 211--249. Springer, 2017.

\bibitem[PW23]{polyanskiy2023information}
Yury Polyanskiy and Yihong Wu.
\newblock {\em Information Theory: From Coding to Learning}.
\newblock Cambridge University Press, 2023+.

\bibitem[PZ21]{pal2021community}
Soumik Pal and Yizhe Zhu.
\newblock Community detection in the sparse hypergraph stochastic block model.
\newblock {\em Random Structures \& Algorithms}, 59(3):407--463, 2021.

\bibitem[Rag16]{raginsky2016strong}
Maxim Raginsky.
\newblock Strong data processing inequalities and $\phi$-sobolev inequalities
  for discrete channels.
\newblock {\em IEEE Transactions on Information Theory}, 62(6):3355--3389,
  2016.

\bibitem[RP19]{roozbehani2019low}
Hajir Roozbehani and Yury Polyanskiy.
\newblock Low density majority codes and the problem of graceful degradation.
\newblock {\em arXiv preprint arXiv:1911.12263}, 2019.

\bibitem[RU08]{richardson2008modern}
Tom Richardson and R{\"u}diger Urbanke.
\newblock {\em Modern coding theory}.
\newblock Cambridge university press, 2008.

\bibitem[Sly09]{sly2009reconstruction}
Allan Sly.
\newblock Reconstruction of random colourings.
\newblock {\em Communications in Mathematical Physics}, 288(3):943--961, 2009.

\bibitem[Sly11]{sly2011reconstruction}
Allan Sly.
\newblock {Reconstruction for the Potts model}.
\newblock {\em The Annals of Probability}, 39(4):1365 -- 1406, 2011.

\bibitem[SM19]{stephan2019robustness}
Ludovic Stephan and Laurent Massouli{\'e}.
\newblock Robustness of spectral methods for community detection.
\newblock In {\em Conference on Learning Theory}, pages 2831--2860. PMLR, 2019.

\bibitem[SM22]{stephan2022non}
Ludovic Stephan and Laurent Massouli{\'e}.
\newblock Non-backtracking spectra of weighted inhomogeneous random graphs.
\newblock {\em Mathematical Statistics and Learning}, 5(3):201--271, 2022.

\bibitem[SR14]{sutter2014universal}
David Sutter and Joseph~M. Renes.
\newblock Universal polar codes for more capable and less-noisy channels and
  sources.
\newblock In {\em 2014 IEEE International Symposium on Information Theory},
  pages 1461--1465. IEEE, 2014.

\bibitem[SZ22]{stephan2022sparse}
Ludovic Stephan and Yizhe Zhu.
\newblock Sparse random hypergraphs: Non-backtracking spectra and community
  detection.
\newblock In {\em 2022 IEEE 63rd Annual Symposium on Foundations of Computer
  Science (FOCS)}, pages 567--575. IEEE, 2022.

\bibitem[ZSTG22]{zhang2022sparse}
Erchuan Zhang, David Suter, Giang Truong, and Syed~Zulqarnain Gilani.
\newblock Sparse hypergraph community detection thresholds in stochastic block
  model.
\newblock In {\em Thirty-Sixth Conference on Neural Information Processing
  Systems (NeurIPS)}, 2022.

\bibitem[ZT22]{zhang2022exact}
Qiaosheng Zhang and Vincent Y.~F. Tan.
\newblock Exact recovery in the general hypergraph stochastic block model.
\newblock {\em IEEE Transactions on Information Theory}, 69(1):453--471, 2022.

\end{thebibliography}
\appendix

\section{Non-reconstruction for the fixed-hypertree BOHT model} \label{sec:boht-non-recon-fixed}
In this section we give a fixed-hypertree version of Theorem~\ref{thm:boht-non-recon}.
Consider the model $\BOHT(T,q,r,\pi,B)$ where $T$ is a fixed rooted $r$-uniform linear hypertree.
Recall the definition of the branching number.
\begin{definition}[Branching number \cite{lyons1990random}]\label{defn:br-tree}
	Let $T$ be a possibly infinite tree rooted at $\rho$.
	Define a flow to be a function $f: V(T) \to \bR_{\ge 0}$ such that
	for every vertex $u$, we have
	\begin{align}
		f_u = \sum_{v\in c(u)} f_v.
	\end{align}
	Define $\br(T)$ to be the $\sup$ of all numbers $\lambda$ such that there exists a flow $f$ with
	$f_\rho > 0$, and $f_u \le \lambda^{-d(u)}$ for all vertices $u$, where $d(u)$ denotes the depth of $u$ (i.e., distance to $\rho$).
\end{definition}
For an $r$-uniform linear hypertree, we can split every downward hyperedge into $(r-1)$ downward edges, and apply the above definition. In this way we extend the definition of branching number to linear hypertrees.
\begin{theorem}[Non-reconstruction for BOHT] \label{thm:boht-non-recon-fixed}
  Consider the model $\BOHT(T,q,r,\pi,B)$.
  \begin{enumerate}[label=(\roman*)]
    \item \label{item:thm-boht-non-recon-fixed:general}
    If
    \begin{align}
      \label{eqn:thm-boht-non-recon-fixed:general-kl} &\br(T) \eta^{(m,ht)}_{\KL}(\pi, B) < 1,\\
      \label{eqn:thm-boht-non-recon-fixed:general-skl} \text{or}\quad & \br(T) \eta^{(m,ht)}_{\SKL}(\pi, B)< 1, I_{\SKL}(\pi, B) < \infty,\text{and $T$ has bounded maximum degree},
    \end{align}
    then reconstruction is impossible.
    \item \label{item:thm-boht-non-recon-fixed:bms}
    Suppose the BOHT model is binary symmetric. If
    \begin{align}
      \label{eqn:thm-boht-non-recon-fixed:bms-kl} & \br(T) \eta^{(m,ht,s)}_{\KL}(B) < 1,\\
      \label{eqn:thm-boht-non-recon-fixed:bms-chi2} \text{or}\quad & \br(T) \eta^{(m,ht,s)}_{\chi^2}(B)< 1,\\
      \label{eqn:thm-boht-non-recon-fixed:bms-skl} \text{or}\quad & \br(T) \eta^{(m,ht,s)}_{\SKL}(B)< 1, C_{\SKL}(B) < \infty,\text{and $T$ has bounded maximum degree},
    \end{align}
    then reconstruction is impossible.
  \end{enumerate}
\end{theorem}
\begin{proof}
  The proof is a generalization of the argument from \cite{gu2020non}.

  \textbf{Part \ref{item:thm-boht-non-recon-fixed:general}, mutual information:}
  For any vertex $u$, let $L_{u,k}$ denote the set of descendants of $u$ at distance $k$ to $\rho$.
  Define
  \begin{align} \label{eqn:proof-thm-boht-non-recon-fixed:general:defn-a}
    a_u := H(\pi)^{-1} \left( \eta^{(m,ht)}_{\KL}(\pi, B) \right)^{d(u)} \lim_{k\to \infty} I(\sigma_u; \sigma_{L_{u,k}}).
  \end{align}
  By DPI, $I(\sigma_u; \sigma_{L_{u,k}})$ is non-increasing for $k\ge d(u)$, so the limit exists.

  Let $\gamma(u)$ denote the set of downward hyperedges of $u$ and $c(u)$ denotes the set of children of $u$.
  For any $e=\{u,v_1,\ldots,v_{r-1}\}\in \gamma(u)$, we have a diagram as in Figure~\ref{fig:proof-thm-boht-non-recon-fixed}, where $P_{\sigma_u}=\pi$, $P_i = P_{\sigma_{L_{v_i,k}} | \sigma_{v_i}}$.
  \begin{figure}[ht]
    \centering
    \includegraphics{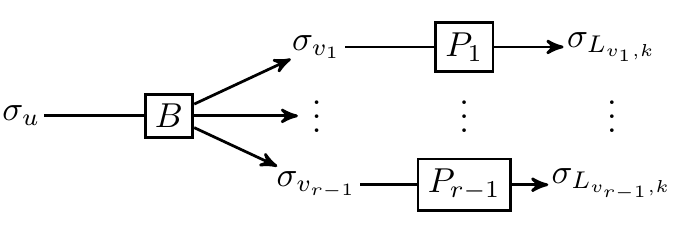}
    \caption{Apply multi-terminal SDPI to BOHT with a fixed hypertree}
    \label{fig:proof-thm-boht-non-recon-fixed}
  \end{figure}
  Define $L_{e\backslash u,k} := \bigcup_{i\in [r-1]} L_{v_i,k}$.
  By definition of multi-terminal contraction coefficients, we have
  \begin{align}
    I(\sigma_u; \sigma_{L_{e\backslash u,k}})
    \le \eta^{(m,ht)}_{\KL}(\pi, B) \sum_{i\in [r-1]} I(\sigma_{v_i}; \sigma_{L_{v_i,k}}).
  \end{align}
  Summing over all $e\in \gamma(u)$ and using subadditivity, we have
  \begin{align} \label{eqn:proof-thm-boht-non-recon-fixed:general:step}
    I(\sigma_u; \sigma_{L_{u,k}})
    \le \sum_{e\in \gamma(u)} I(\sigma_u; \sigma_{L_{e\backslash u,k}})
    \le \eta^{(m,ht)}_{\KL}(\pi, B) \sum_{v\in c(u)} I(\sigma_v; \sigma_{L_{v,k}}).
  \end{align}
  Comparing \eqref{eqn:proof-thm-boht-non-recon-fixed:general:defn-a} and \eqref{eqn:proof-thm-boht-non-recon-fixed:general:step} we see that
  \begin{align}
    a_u \le \sum_{v\in c(u)} a_v.
  \end{align}
  By definition, we have
  \begin{align}
    a_u \le \left( \eta^{(m,ht)}_{\KL}(\pi, B) \right)^{d(u)}.
  \end{align}
  However, $a$ is not a flow yet. We define a flow $b$ from $a$.
  For a vertex $u$, let $u_0=\rho, \ldots, u_\ell=u$ be the shortest path from $\rho$ to $u$.
  Define
  \begin{align}
    b_u := a_u \prod_{0\le j\le \ell-1} \frac{a_{u_j}}{\sum_{v\in c(u_j)} a_v}.
  \end{align}
  (If the denominator is zero, let $b_u=0$.) Then we have
  \begin{align}
    b_u = \sum_{v\in c(u)} b_v, \quad \text{and} \quad b_u \le a_u \le \left( \eta^{(m,ht)}_{\KL}(\pi, B) \right)^{d(u)}.
  \end{align}
  By definition of branching number, we have $b_\rho=0$. Therefore
  \begin{align}
    \lim_{k\to \infty} I(\sigma_\rho; \sigma_{L_k})=0
  \end{align}
  and non-reconstruction holds.

  \textbf{Part \ref{item:thm-boht-non-recon-fixed:general}, SKL information:}
  Suppose every vertex $u$ has at most $\gamma_{\max}$ downward hyperedges. Then
  \begin{align} \label{eqn:proof-thm-boht-non-recon-fixed:general:max-skl}
    \lim_{k\to \infty} I_{\SKL}(\sigma_u; \sigma_{L_{u,k}})
    \le I_{\SKL}(\sigma_u; \sigma_{L_{u,d(u)+1}})
    = |\gamma(u)| I_{\SKL}(\pi, B)
    \le \gamma_{\max} I_{\SKL}(\pi, B) < \infty.
  \end{align}
  We define
  \begin{align}
    a_u := \left(\gamma_{\max} I_{\SKL}(\pi, B)\right)^{-1} \left( \eta^{(m,ht)}_{\SKL}(\pi, B) \right)^{d(u)} \lim_{k\to \infty} I_{\SKL}(\sigma_u; \sigma_{L_{u,k}}).
  \end{align}
  By \eqref{eqn:proof-thm-boht-non-recon-fixed:general:max-skl}, we have
  \begin{align}
    a_u \le \left( \eta^{(m,ht)}_{\SKL}(\pi, B) \right)^{d(u)}.
  \end{align}
  The rest of the proof is similar to the mutual information case.

  \textbf{Part \ref{item:thm-boht-non-recon-fixed:bms}, KL and SKL capacity:}
  In this case, channels appeared in the above proof (e.g., $\sigma_u \mapsto \sigma_{L_{u,k}}$) are all BMS channels. Using the same proof with BMS version of multi-terminal contraction coefficients leads to the desired result.

  \textbf{Part \ref{item:thm-boht-non-recon-fixed:bms}, $\chi^2$-capacity:}
  Note that $C_{\chi^2}(P) \le 1$ for all BMS channels $P$. So we can define
  \begin{align}
    a_u := \left( \eta^{(m,ht,s)}_{\chi^2}(B) \right)^{d(u)} \lim_{k\to \infty} C_{\chi^2}(\sigma_{L_{u,k}})
  \end{align}
  and it satisfies
  \begin{align}
    a_u \le \left( \eta^{(m,ht,s)}_{\chi^2}(\pi, B) \right)^{d(u)}.
  \end{align}
  The rest of the proof is similar to the previous cases.
\end{proof}

\section{Computation of SKL contraction coefficients} \label{sec:skl-multi-sdpi}
In this section we prove Prop.~\ref{prop:skl-multi-sdpi}, which says that for any BMS channels $P_1,\ldots,P_{r-1}$ we have
\begin{align} \label{eqn:skl-contraction-edge}
  C_{\SKL}((P_1\times \cdots \times P_{r-1}) \circ B_{r,\lambda}) \le \lambda^2 \sum_{i\in [r-1]} C_{\SKL}(P_i).
\end{align}

By BSC mixture representation of BMS channels, \eqref{eqn:skl-contraction-edge} is equivalent to
\begin{align} \label{eqn:skl-contraction-edge-bsc}
  C_{\SKL}\left((\BSC_{\Delta_1} \times \cdots \times \BSC_{\Delta_{r-1}}) \circ B_{r,\lambda}\right) \le \lambda^2 \sum_{i\in [r-1]} C_{\SKL}(\BSC_{\Delta_i})
\end{align}
for all $\Delta_1,\ldots,\Delta_{r-1} \in \left[0,\frac 12\right]$.

\subsection{Case \texorpdfstring{$r=3$}{r=3}} \label{sec:skl-multi-sdpi:r3}
We first prove the $r=3$ case.
\begin{lemma} \label{lem:skl-contraction-edge-bsc-r3}
For any $\Delta_1,\Delta_2\in\left[0,\frac 12\right]$, we have
\begin{align} \label{eqn:skl-contraction-edge-bsc-r3}
  C_{\SKL}((\BSC_{\Delta_1} \times \BSC_{\Delta_2}) \circ B_{r,\lambda}) \le \lambda^2 (C_{\SKL}(\BSC_{\Delta_1}) + C_{\SKL}(\BSC_{\Delta_2})).
\end{align}
Furthermore, the inequality is strict when $0<\lambda<1$ and $\min\{\Delta_1,\Delta_2\}<\frac 12$.
\end{lemma}
\begin{proof}
  We expand LHS of~\eqref{eqn:skl-contraction-edge-bsc-r3} using the BP recursion formula established in Section~\ref{sec:non-recon-special}.
  Let $\theta_i=1-2\Delta_i$ for $i=1,2$. Then
  \begin{align} \label{eqn:skl-contraction-edge-bsc-r3-1}
    & C_{\SKL}((\BSC_{\Delta_1} \times \BSC_{\Delta_2})\circ B_{r,\lambda}) \\
    \nonumber =&~ \sum_{x_1=+,x_2\in \{\pm\}}
    \frac 12 \lambda (\theta_1 x_1 + \theta_2 x_2)
    \arctanh\frac{\lambda(\theta_1 x_1 + \theta_2 x_2)}{\lambda(1 + \theta_1 x_1 \theta_2 x_2) + (1-\lambda)} \\
    \nonumber =&~ \lambda \left(\frac 12 (\theta_1+\theta_2) \arctanh\frac{\lambda(\theta_1+\theta_2)}{1+\lambda \theta_1 \theta_2} + \frac 12 (\theta_1-\theta_2) \arctanh\frac{\lambda(\theta_1-\theta_2)}{1-\lambda \theta_1 \theta_2}\right) \\
    \nonumber =&~ \lambda \left(\frac 12 (\theta_1+\theta_2)F_\lambda(\theta_1,\theta_2) + \frac 12 (\theta_1-\theta_2) F_\lambda(\theta_1,-\theta_2)\right)
  \end{align}
  where
  \begin{align}
    F_\lambda(\theta_1,\theta_2) := \arctanh\frac{\lambda(\theta_1 + \theta_2)}{1+\lambda \theta_1 \theta_2}.
  \end{align}
  Note that by definition, $F_\lambda(\theta_1,\theta_2) = -F_\lambda(-\theta_1,-\theta_2)$
  and $F_\lambda(\theta_1,\theta_2) = F_\lambda(\theta_2,\theta_1)$.

  We have
  \begin{align} \label{eqn:skl-contraction-edge-bsc-r3-2}
    &~ \frac 12 (\theta_1+\theta_2)F_\lambda(\theta_1,\theta_2) + \frac 12 (\theta_1-\theta_2) F_\lambda(\theta_1,-\theta_2)  \\
    \nonumber =&~ \frac 12 \theta_1 (F_\lambda(\theta_1,\theta_2)+F_\lambda(\theta_1,-\theta_2))
    + \frac12 \theta_2 (F_\lambda(\theta_1,\theta_2)+F_\lambda(-\theta_1,\theta_2))\\
    \nonumber \le&~ \theta_1 F_\lambda(\theta_1,0) + \theta_2 F_\lambda(0,\theta_2) \\
    \nonumber =&~ \theta_1 \arctanh(\lambda\theta_1) + \theta_2 \arctanh(\lambda\theta_2) \\
    \nonumber \le&~ \lambda(\theta_1 \arctanh \theta_1 + \theta_2 \arctanh \theta_2) \\
    \nonumber =&~ \lambda (C_{\SKL}(\BSC_{\Delta_1}) + C_{\SKL}(\BSC_{\Delta_2})),
  \end{align}
  where the second step follows from Lemma~\ref{lem:boht-tech-r3}, and the fourth step follows convexity of $\arctanh$ in $[0,1]$.
  Combining~\eqref{eqn:skl-contraction-edge-bsc-r3-1}\eqref{eqn:skl-contraction-edge-bsc-r3-2} we finish the proof.
\end{proof}

\begin{lemma} \label{lem:boht-tech-r3}
For $\lambda,\theta_1,\theta_2 \in [0, 1]$, we have
\begin{align} \label{eqn:lem-boht-tech-r3}
\frac 12 \left(F_\lambda(\theta_1,\theta_2) + F_\lambda(\theta_1,-\theta_2)\right) \le F_\lambda(\theta_1,0).
\end{align}
Furthermore, the inequality is strict when $0<\lambda<1$ and $\theta_1,\theta_2>0$.
\end{lemma}

\begin{proof}
We use the formula
\begin{align}
  \arctanh x + \arctanh y = \arctanh \frac{x+y}{1+x y}
\end{align}
to expand both sides of~\eqref{eqn:lem-boht-tech-r3}. LHS is
\begin{align} \label{eqn:lem-boht-tech-r3-lhs}
  & F_\lambda(\theta_1,\theta_2) + F_\lambda(\theta_1,-\theta_2) \\
  =&~ \arctanh\frac{\lambda(\theta_1 + \theta_2)}{1+\lambda \theta_1 \theta_2} + \arctanh\frac{\lambda(\theta_1 - \theta_2)}{1-\lambda \theta_1 \theta_2} \nonumber \\
  =&~ \arctanh \frac{2 \lambda \theta_1 (1-\lambda \theta_2^2)}{\lambda^2(\theta_1^2-\theta_2^2) + 1-\lambda^2 \theta_1^2 \theta_2^2}. \nonumber
\end{align}
RHS is
\begin{align} \label{eqn:lem-boht-tech-r3-rhs}
  2 F_\lambda(\theta_1,0) = \arctanh \frac{2\lambda \theta_1}{1+\lambda^2\theta_1^2}.
\end{align}
By comparing \eqref{eqn:lem-boht-tech-r3-lhs}\eqref{eqn:lem-boht-tech-r3-rhs} and using monotonicity of $\arctanh$, it suffices to prove that
\begin{align}
  \frac{1-\lambda \theta_2^2}{\lambda^2(\theta_1^2-\theta_2^2) + 1-\lambda^2 \theta_1^2 \theta_2^2} \le \frac{1}{1+\lambda^2\theta_1^2}.
\end{align}
We have
\begin{align}
  (\lambda^2(\theta_1^2-\theta_2^2) + 1-\lambda^2 \theta_1^2 \theta_2^2) - (1-\lambda \theta_2^2)(1+\lambda^2\theta_1^2) = \lambda(1-\lambda)(1-\lambda \theta_1^2) \theta_2^2 \ge 0.
\end{align}
This finishes the proof.
\end{proof}

\begin{proof}[Proof of Prop.~\ref{prop:skl-multi-sdpi}, case $r=3$]
  By BSC mixture representation of BMS channels (Lemma~\ref{lem:bms-mixture}) and Lemma~\ref{lem:skl-contraction-edge-bsc-r3}.
\end{proof}

\subsection{Case \texorpdfstring{$r=4$}{r=4}} \label{sec:skl-multi-sdpi:r4}
Now we prove the $r=4$ case.
\begin{lemma} \label{lem:skl-contraction-edge-bsc-r4}
For all $\Delta_1,\Delta_2,\Delta_3\in \left[0,\frac 12\right]$, we have
\begin{align} \label{eqn:skl-contraction-edge-bsc-r4}
  C_{\SKL}((\BSC_{\Delta_1} \times \BSC_{\Delta_2} \times \BSC_{\Delta_3}) \circ B_{r,\lambda}) \le \lambda^2 \sum_{i\in [3]} C_{\SKL}(\BSC_{\Delta_i}).
\end{align}
Furthermore, the inequality is strict when $0<\lambda<1$ and $\min\{\Delta_1,\Delta_2,\Delta_3\} < \frac 12$.
\end{lemma}
The proof is based on the following inequality. In a preliminary version of the current paper, we proposed this inequality as a conjecture based on numerical computation. Shortly after that, \cite{marwaha2023useful} gave a beautiful analytical proof.
\begin{lemma}[\cite{marwaha2023useful}] \label{lem:boht-r4}
  For $\lambda,\theta_1,\theta_2,\theta_3\in [0,1]$, we have
  \begin{align}
    &~ \frac 14 (G_\lambda(\theta_1,\theta_2,\theta_3) + G_\lambda(\theta_1,-\theta_2,\theta_3) + G_\lambda(\theta_1,\theta_2,-\theta_3) + G_\lambda(\theta_1,-\theta_2,-\theta_3)) \\
    \le&~ \lambda \sum_{i\in [3]} \theta_i \arctanh \theta_i, \nonumber
  \end{align}
  where
  \begin{align}
    G_\lambda(\theta_1,\theta_2,\theta_3)&:= (\theta_1+\theta_2+\theta_3+\theta_1\theta_2\theta_3) F_\lambda(\theta_1,\theta_2,\theta_3), \\
    F_\lambda(\theta_1,\theta_2,\theta_3)&:= \arctanh \frac{\lambda(\theta_1+\theta_2+\theta_3+\theta_1\theta_2\theta_3)}{1+\lambda(\theta_1\theta_2+\theta_2\theta_3+\theta_3\theta_1)}.
  \end{align}
  Furthermore, the inequality is strict when $0<\lambda<1$ and $\max\{\theta_1,\theta_2,\theta_3\} > 0$.
\end{lemma}
\begin{proof}[Proof of Lemma~\ref{lem:skl-contraction-edge-bsc-r4}]
  We expand LHS of~\eqref{eqn:skl-contraction-edge-bsc-r4} using BP recursion formula established in Section~\ref{sec:non-recon-special}.
  Let $\theta_i=1-2\Delta_i$ for $i\in [3]$. Then
  \begin{align} \label{eqn:skl-contraction-edge-bsc-r4-1}
    & C_{\SKL}((\BSC_{\Delta_1} \times \BSC_{\Delta_2}\times \BSC_{\Delta_3})\circ B_{r,\lambda}) \\
    \nonumber =&~ \sum_{x_1=+,x_2,x_3\in \{\pm\}}
    \frac 14 \lambda (\theta_1 x_1 + \theta_2 x_2 + \theta_3 x_3 + \theta_1 \theta_2 \theta_3 x_1 x_2 x_3) \\
    \nonumber & \cdot \arctanh\frac{\lambda(\theta_1 x_1 + \theta_2 x_2 + \theta_3 x_3 + \theta_1 \theta_2 \theta_3 x_1 x_2 x_3)}{1 + \lambda (\theta_1 x_1 \theta_2 x_2 + \theta_2 x_3 \theta_3 x_3 + \theta_3 x_3 \theta_1 x_1)} \\
    \nonumber =&~ \frac \lambda 4(G_\lambda(\theta_1,\theta_2,\theta_3) + G_\lambda(\theta_1,-\theta_2,\theta_3) + G_\lambda(\theta_1,\theta_2,-\theta_3) + G_\lambda(\theta_1,-\theta_2,-\theta_3))
  \end{align}
  where
  \begin{align}
    G_\lambda(\theta_1,\theta_2,\theta_3)&:= (\theta_1+\theta_2+\theta_3+\theta_1\theta_2\theta_3) F_\lambda(\theta_1,\theta_2,\theta_3), \\
    F_\lambda(\theta_1,\theta_2,\theta_3)&:= \arctanh \frac{\lambda(\theta_1+\theta_2+\theta_3+\theta_1\theta_2\theta_3)}{1+\lambda(\theta_1\theta_2+\theta_2\theta_3+\theta_3\theta_1)}.
  \end{align}
  By Lemma~\ref{lem:boht-r4}, we have
  \begin{align} \label{eqn:skl-contraction-edge-bsc-r4-2}
    &~ \frac 14 (G_\lambda(\theta_1,\theta_2,\theta_3) + G_\lambda(\theta_1,-\theta_2,\theta_3) + G_\lambda(\theta_1,\theta_2,-\theta_3) + G_\lambda(\theta_1,-\theta_2,-\theta_3)) \\
    \nonumber \le&~ \lambda \sum_{i\in[3]} \theta_i \arctanh \theta_i \\
    \nonumber =&~ \lambda \sum_{i\in [3]} C_{\SKL}(\BSC_{\Delta_i}).
  \end{align}
  Combining~\eqref{eqn:skl-contraction-edge-bsc-r4-1}\eqref{eqn:skl-contraction-edge-bsc-r4-2} we finish the proof.
\end{proof}

\begin{proof}[Proof of Prop.~\ref{prop:skl-multi-sdpi}, case $r=4$]
  By BSC mixture representation of BMS channels (Lemma~\ref{lem:bms-mixture}) and Lemma~\ref{lem:skl-contraction-edge-bsc-r4}.
\end{proof}


\subsection{Handle the critical case} \label{sec:skl-multi-sdpi:critical}
In section we prove the critical case of Theorem~\ref{thm:boht-non-recon-special}\ref{item:thm-boht-non-recon-special-r34}, that for $r=3,4$ and $(r-1)d\lambda^2=1$, reconstruction is impossible for the BOHT model $\BOHT(2,r,\lambda,D)$.

Before giving the proof, we introduce some more preliminaries on BMS channels.
\begin{definition}[Limit of BMS channels] \label{defn:bms-limit}
  Let $(P_k)_{k\ge 0}$ be a sequence of BMS channels and $P_\infty$ be a BMS channel.
  For $k\in \bZ_{\ge 0} \cup \{\infty\}$, let $\Delta_k\in\left[0,\frac 12\right]$ denote the $\Delta$-component of $P_k$ and $P_{\Delta_k} \in \cP\left(\left[0,\frac 12\right]\right)$ denote its distribution.
  We say $(P_k)_{k\ge 0}$ converges weakly to $P_\infty$ if $(P_{\Delta_k})_{k\to \infty}$ converges weakly to $P_{\Delta_\infty}$ as distributions on $\left[0,\frac 12\right]$.
\end{definition}

The following lemma is a direct consequence of \cite[Lemma 11.2]{gu2023channel} (see also \cite[Theorem 7.24]{richardson2008modern}).
\begin{lemma} \label{lem:bms-down-seq-limit}
  Let $(P_k)_{k\ge 0}$ be a sequence of BMS channels.
  If $P_{k+1}\le_{\deg} P_k$ for all $k\ge 0$, then $(P_k)_{k\ge 0}$ converges weakly to some BMS channels $P_\infty$.
\end{lemma}

Recall that $M_k$ is the BMS channel $\sigma_\rho \mapsto (T_k, \sigma_{L_k})$.
Then $M_{k+1}\le_{\deg} M_k$. By Lemma~\ref{lem:bms-down-seq-limit}, the limit $M_\infty := \lim_{k\to \infty} M_k$ exists.
For $k\in \bZ_{\ge 0} \cup \{\infty\}$, let $\Delta_k$ denote the $\Delta$-component of $M_k$.
Then $P_{\Delta_k}$ converges weakly to $P_{\Delta_\infty}$ as $k\to \infty$.

The proof idea is as follows. By definition of the limit, we have $\BP(M_\infty) = M_\infty$.
If $M_\infty$ is non-trivial, then by Eq.~\eqref{eqn:prop-skl-multi-sdpi-strict}, we have $C_{\SKL}(\BP(M_\infty)) < C_{\SKL}(M_\infty)$, which leads to contradiction.
The actual argument is more involved because SKL capacity can be infinite for some BMS channels.

\begin{proof}[Proof of Theorem~\ref{thm:boht-non-recon-special}\ref{item:thm-boht-non-recon-special-r34} critical case]
The case $\lambda=1$ is already handled in Section~\ref{sec:non-recon-special}.
So we can wlog assume $0<\lambda<1$.

Suppose for the sake of contradiction that reconstruction holds.
Then the limit channel $M_\infty$ is non-trivial.

We first prove that $\bP[\Delta_\infty=0]=0$. If not, then by weak convergence, there exists $\delta>0$ such that for all $\epsilon>0$, $\lim_{k\to\infty}\bP[\Delta_k < \epsilon] > \delta$.
Because $\lim_{\epsilon\to 0^+} C_{\SKL}(\BSC_\epsilon) = \infty$, this implies
$\lim_{k\to \infty} C_{\SKL}(M_k) = \infty$.
However, for all $k\ge 1$, we have $C_{\SKL}(M_k) \le C_{\SKL}(M_1) = d C_{\SKL}(B_{r,\lambda}) < \infty$. Contradiction. So $\bP[\Delta_\infty=0]=0$.

Because $M_\infty$ is non-trivial, $\bP\left[\Delta_\infty=\frac 12\right]<1$.
So $\bP\left[0<\Delta_\infty<\frac 12\right] > 0$.
By weak convergence, there exists $c>0$ and a closed interval $I \subseteq \left(0,\frac 12\right)$ such that $\bP[\Delta_k\in I] \ge c$ for $k$ large enough.
By Prop.~\ref{prop:skl-multi-sdpi}, $\forall \delta_1,\ldots,\delta_{r-1}\in I$, we have
\begin{align}
  C_{\SKL}((\BSC_{\delta_1}\times \cdots \times \BSC_{\delta_{r-1}})\circ B_{r,\lambda}) < \lambda^2 \sum_{i\in [r-1]} C_{\SKL}(\BSC_{\delta_i}).
\end{align}
Because $I$ is compact, there exists $\epsilon>0$ such that $\forall \delta_1,\ldots,\delta_{r-1}\in I$,
\begin{align}
  C_{\SKL}((\BSC_{\delta_1}\times \cdots \times \BSC_{\delta_{r-1}})\circ B_{r,\lambda}) \le \lambda^2 \sum_{i\in [r-1]} C_{\SKL}(\BSC_{\delta_i}) - \epsilon.
\end{align}
For any $k$ satisfying $\bP[\Delta_k\in I] \ge c$, we have
\begin{align}
  &~C_{\SKL}(\BP(M_k)) \\
  \nonumber =&~ d C_{\SKL}\left(M_k^{\times (r-1)}\circ B_{r,\lambda}\right) \\
  \nonumber =&~ d \underset{\delta_1,\ldots,\delta_{r-1}\sim P_{\Delta_k}}{\bE} C_{\SKL}((\BSC_{\delta_1}\times \cdots \times \BSC_{\delta_{r-1}})\circ B_{r,\lambda}) \\
  \nonumber \le&~ d \underset{\delta_1,\ldots,\delta_{r-1}\sim P_{\Delta_k}}{\bE} \left[ \lambda^2 \sum_{i\in [r-1]} C_{\SKL}(\BSC_{\delta_i}) - \epsilon \mathbbm{1}\{\delta_1,\ldots,\delta_{r-1}\in I\}\right] \\
  \nonumber \le&~ d ((r-1)\lambda^2 C_{\SKL}(M_k) - \epsilon c^{r-1}) \\
  \nonumber =&~ C_{\SKL}(M_k) - d \epsilon c^{r-1}.
\end{align}
So for $k$ large enough,
\begin{align}
  C_{\SKL}(M_{k+1}) \le C_{\SKL}(M_k) - d \epsilon c^{r-1}.
\end{align}
Because $d \epsilon c^{r-1}>0$ and $C_{\SKL}(M_k)<\infty$ for all $k\ge 1$, this implies $C_{\SKL}(M_k)<0$ for $k$ large enough, which cannot be true.
This finishes the proof.
\end{proof}

\section{Computation of \texorpdfstring{$\chi^2$}{chi2}-contraction coefficients} \label{sec:chi2-multi-sdpi}
In this section we prove Prop.~\ref{prop:chi2-multi-sdpi} and Lemma~\ref{lem:chi2-multi-sdpi-large-lambda}, which computes the $\chi^2$-multi-terminal contraction coefficient $\eta^{(m,s)}_{\chi^2}(B_{r,\lambda})$.
In fact, our method works for the more general setting where $B: \{\pm\}\to \{\pm\}^{r-1}$ together with the sign flip $\{\pm\}^{r-1}\to \{\pm\}^{r-1}$ is a BMS channel.

\subsection{Less-noisy preorder} \label{sec:chi2-multi-sdpi:less-noisy}
Our method uses the less-noisy preorder, a very useful channel preorder, especially for BMS channels.
\begin{definition}[Less-noisy preorder \cite{korner1977comparison}] \label{defn:less-noisy}
  Let $P: \cX \to \cY$ and $Q: \cX \to \cZ$ be two channels with the same input alphabet. We say $P$ is less noisy than $Q$, denoted $P\ge_{\ln} Q$, if for every measurable space $\cW$, distribution $\pi\in \cP(\cW)$, and channel $R: \cW \to \cX$, we have $I(\pi, P\circ R) \ge I(\pi, Q\circ R)$.
\end{definition}
Less-noisy preorder behaves nicely under channel transformations, summarized as follows.
\begin{itemize}
  \item (Composition) Let $P, Q$ be two channels with the same input alphabet $\cX$. Let $R: \cW\to \cX$ be a channel. If $P \le_{\ln} Q$, then $P\circ R\le_{\ln} Q\circ R$.
  \item (Tensorization) \cite{sutter2014universal,polyanskiy2017strong} Let $P_1$ and $Q_1$ be two channels with the same input alphabet $\cX$, and $P_2$ and $Q_2$ be two channels with the same input alphabet $\cY$. If $P_1\le_{\ln} Q_1$ and $P_2\le_{\ln} Q_2$, then $P_1\times Q_1 \le_{\ln} P_2 \times Q_2$.
  \item ($\star$-convolution) Let $P_1,P_2,Q_1,Q_2$ be four channels with the same input alphabet. If $P_1\le_{\ln} Q_1$ and $P_2\le_{\ln} Q_2$, then $P_1\star Q_1 \le_{\ln} P_2 \star Q_2$.
\end{itemize}

Although defined using mutual information, less-noisy preorder is closely related to $\chi^2$-divergence. \cite[Theorem 1]{makur2018comparison} implies that for BMS channels $P, Q$, if $P \le_{\ln} Q$, then $C_{\chi^2}(P) \le C_{\chi^2}(Q)$.
Furthermore, under $\chi^2$-capacity constraint, BEC and BSC are the extremal channels in less-noisy preorder.
\begin{lemma}[{\cite[Lemma 2]{roozbehani2019low}}]\label{lem:bms-extremal-less-noisy}
  Among all BMS channels with the same $\chi^2$-capacity $C_{\chi^2}(W)=\eta$
  the least noisy one is $\BEC$ and the most noisy one is $\BSC$, i.e.
  \begin{equation}\label{eqn:bms-extremal-ln}
    \BSC_{1/2-\sqrt{\eta}/2} \le_{\ln} W \le_{\ln} \BEC_{1-\eta}.
  \end{equation}
\end{lemma}

\subsection{Binary symmetric model} \label{sec:chi2-multi-sdpi:bms-boht}
We consider the more general setting where $B: \{\pm\}\to \{\pm\}^{r-1}$ is a BMS channel. Recall that this is the setting for the binary symmetric BOHT model.
\begin{proposition} \label{prop:chi2-multi-sdpi-bms}
  Suppose $B: \{\pm\} \to \{\pm\}^{r-1}$ together with the sign flip $\{\pm\}^{r-1}\to \{\pm\}^{r-1}$ is a BMS channel.
  Then
  \begin{align}
    \eta^{(m,s)}_{\chi^2}(B) = \sup_{0<\epsilon\le 1} f_B(\epsilon)
  \end{align}
  where
  \begin{align} \label{eqn:prop-chi2-multi-sdpi-bms:f_B}
    f_B(\epsilon) := \frac 1{(r-1)\epsilon} C_{\chi^2}\left(\BEC_{1-\epsilon}^{\times (r-1)} \circ B\right)
  \end{align}
  is a polynomial of degree $r-2$.
\end{proposition}
\begin{proof}
  Let $P$ be a non-trivial BMS channel and $\epsilon=C_{\chi^2}(P)$.
  By Lemma~\ref{lem:bms-extremal-less-noisy}, $P\le_{\ln} \BEC_{1-\epsilon}$.
  Because less-noisy preorder is preserved under tensorization and composition, we have
  \begin{align}
    P^{\times (r-1)}\circ B \le_{\ln} \BEC_{1-\epsilon}^{\times (r-1)} \circ B.
  \end{align}
  Then by \cite[Theorem 1]{makur2018comparison} we have
  \begin{align}
    C_{\chi^2}(P^{\times (r-1)}\circ B) \le C_{\chi^2}\left(\BEC_{1-\epsilon}^{\times (r-1)} \circ B\right).
  \end{align}
  So
  \begin{align}
    \eta^{(m,s)}_{\chi^2}(B) = \sup_{P} \frac{C_{\chi^2}(P^{\times (r-1)}\circ B)}{(r-1)C_{\chi^2}(P)}
    = \sup_{0<\epsilon\le 1} \frac{C_{\chi^2}\left(\BEC_{1-\epsilon}^{\times (r-1)} \circ B\right)}{(r-1)C_{\chi^2}(\BEC_{1-\epsilon})}
    = \sup_{0<\epsilon\le 1} f_B(\epsilon).
  \end{align}

  It remains to prove that $f_B(\epsilon)$ is a polynomial of degree of $r-2$.
  By BSC mixture representation, we have
  \begin{align}
    \epsilon f_B(\epsilon) &= \frac 1{r-1}C_{\chi^2}\left(\BEC_{1-\epsilon}^{\times (r-1)}\circ B\right)\\
    \nonumber &= \frac 1{r-1}\sum_{x\in \{0,1\}^{r-1}} \left(\prod_{i\in [r-1]} \left(\epsilon^{x_i}(1-\epsilon)^{1-x_i}\right)\right) C_{\chi^2}\left( \left(\prod_{i\in [r-1]} \BSC_{(1-x_i)/2}\right) \circ B\right),
  \end{align}
  which is a degree-$(r-1)$ polynomial.
  Furthermore, the constant coefficient of $\epsilon f_B(\epsilon)$ is
  \begin{align}
    \frac 1{r-1}C_{\chi^2}\left(\BSC_{1/2}^{\times (r-1)}\circ B\right)=0.
  \end{align}
  So $f_B(\epsilon)$ is a polynomial of degree $r-2$.
\end{proof}

For $r=2$, $f_B(\epsilon)$ is a constant, and we get $\eta^{(m,s)}_{\chi^2}(B) = f_B(1) = C_{\chi^2}(B)$.
For $r=3$, $f_B(\epsilon)$ is a linear function, so $\eta^{(m,s)}_{\chi^2}(B) = \max\{f_B(0), f_B(1)\}$.

Prop.~\ref{prop:chi2-multi-sdpi-bms} immediately leads to the following corollary.
\begin{corollary}[Non-reconstruction for binary symmetric BOHT] \label{coro:boht-non-recon-bms}
  Consider a binary symmetric BOHT model $\BOHT(2,r,\Unif(\{\pm\}),B,D)$ where $\bE_{t\sim D}t=d$.
  If
  \begin{align}
    (r-1) d \sup_{0<\epsilon\le 1} f_B(\epsilon) < 1,
  \end{align}
  where $f_B$ is defined in \eqref{eqn:prop-chi2-multi-sdpi-bms:f_B}, then reconstruction is impossible.
\end{corollary}
\begin{proof}
  By Theorem~\ref{thm:boht-non-recon}\ref{item:thm-boht-non-recon:bms} and Prop.~\ref{prop:chi2-multi-sdpi-bms}.
\end{proof}

\subsection{Special BOHT model} \label{sec:chi2-multi-sdpi:special-boht}
We apply Prop.~\ref{prop:chi2-multi-sdpi-bms} to the special case where $B=B_{r,\lambda}$.
\begin{proof}[Proof of Prop.~\ref{prop:chi2-multi-sdpi}]
  It suffices to prove that $f_{B_{r,\lambda}} = f_{r,\lambda}$, where LHS is defined in \eqref{eqn:prop-chi2-multi-sdpi-bms:f_B} and RHS is defined in \eqref{eqn:thm-boht-non-recon-special:chi2-f}.
  We have
  \begin{align}
    &~f_{B_{r,\lambda}}(\epsilon) \\
    \nonumber =&~ \frac 1{(r-1)\epsilon}\sum_{x\in \{0,1\}^{r-1}} \left(\prod_{i\in [r-1]} \left(\epsilon^{x_i}(1-\epsilon)^{1-x_i}\right)\right) C_{\chi^2}\left( \left(\prod_{i\in [r-1]} \BSC_{(1-x_i)/2}\right) \circ B_{r,\lambda}\right) \\
    \nonumber =&~ \frac 1{(r-1)\epsilon}\sum_{1\le i\le r-1} \binom{r-1}i \epsilon^i (1-\epsilon)^{r-1-i} C_{\chi^2}\left( \left(\Id^{\times i} \times 0^{\times (r-1-i)}\right) \circ B_{r,\lambda}\right) \\
    \nonumber =&~ \frac 1{r-1}\sum_{1\le i\le r-1} \binom{r-1}i \epsilon^{i-1} (1-\epsilon)^{r-1-i} C_{\chi^2}\left( \left(\Id^{\times i} \times 0^{\times (r-1-i)}\right) \circ B_{r,\lambda}\right),
  \end{align}
  where $\Id=\BSC_0$ denotes the identity channel and $0=\BSC_{1/2}$ denotes the trivial channel.
  Using the BP recursion formula established in Section~\ref{sec:non-recon-special}, we have
  \begin{align}
    C_{\chi^2}\left( \left(\Id^{\times i} \times 0^{\times (r-1-i)}\right) \circ B_{r,\lambda}\right)
    =\frac{\lambda^2}{\lambda+(1-\lambda)2^{1-i}}
  \end{align}
  for $1\le i\le r-1$.
  Therefore
  \begin{align}
    f_{B_{r,\lambda}}(\epsilon) = \frac 1{r-1}\sum_{1\le i\le r-1} \binom{r-1}i \epsilon^{i-1} (1-\epsilon)^{r-1-i}  \frac{\lambda^2}{\lambda+(1-\lambda)2^{1-i}} = f_{r,\lambda}(\epsilon).
  \end{align}
  This finishes the proof.
\end{proof}

\begin{proof}[Proof of Lemma~\ref{lem:chi2-multi-sdpi-large-lambda}]
  Note that $f_{r,\lambda}(0) = \lambda^2$.
  Let us compute $f_{r,\lambda}'(\epsilon)$.
  \begin{align}
    f_{r,\lambda}'(\epsilon) =&~ \frac 1{r-1} \sum_{1\le i\le r-1} \binom{r-1}i\left(\frac {d}{d\epsilon} \left((1-\epsilon)^{r-1-i} \epsilon^{i-1}\right)\right) \frac{\lambda^2}{\lambda+(1-\lambda)2^{1-i}} \\
    \nonumber =&~ \frac 1{r-1} \sum_{1\le i\le r-1} \binom{r-1}i \left((i-1)(1-\epsilon)^{r-1-i} \epsilon^{i-2}\right.\\
    \nonumber &~\left.- (r-1-i) (1-\epsilon)^{r-2-i} \epsilon^{i-1} \right) \frac{\lambda^2}{\lambda+(1-\lambda)2^{1-i}}\\
    \nonumber =&~ \frac 1{r-1} \sum_{1\le i\le r-1} \binom{r-1}i (i-1)(1-\epsilon)^{r-1-i} \epsilon^{i-2} \frac{\lambda^2}{\lambda+(1-\lambda)2^{1-i}} \\
    \nonumber &~- \frac 1{r-1} \sum_{2\le i\le r} \binom{r-1}{i-1} (r-i)(1-\epsilon)^{r-1-i} \epsilon^{i-2} \frac{\lambda^2}{\lambda+(1-\lambda)2^{2-i}} \\
    \nonumber =&~ \frac 1{r-1} \sum_{2\le i\le r-1} \binom{r-1}i (1-\epsilon)^{r-1-i} \epsilon^{i-2} \left( \frac{(i-1) \lambda^2}{\lambda+(1-\lambda)2^{1-i}} - \frac{i \lambda^2}{\lambda+(1-\lambda)2^{2-i}} \right).
  \end{align}
  When $\lambda\in \left[\frac 15, 1\right]$, we have
  \begin{align}
    \frac{i-1}{\lambda+(1-\lambda)2^{1-i}} \le \frac{i}{\lambda+(1-\lambda)2^{2-i}}
  \end{align}
  for all integer $i\ge 2$, and the inequality is strict for $i\in \{2\}\cup \bZ_{\ge 5}$.
  Therefore $f_{r,\lambda}'(\epsilon) < 0$ for all $\epsilon\in[0,1]$.
  So for $0<\epsilon\le 1$ we have $f_{r,\lambda}(\epsilon) < f_{r,\lambda}(0) = \lambda^2$.
\end{proof}
We remark that for fixed $r$, the range $\lambda\in\left[\frac 15,1\right]$ could be improved. For example, for $r=5$ and $\lambda\in\left[\frac 17,1\right]$ we have $f_{r,\lambda}(\epsilon)\le \lambda^2$ for all $\epsilon\in[0,1]$.

\subsection{Handle the critical case} \label{sec:chi2-multi-sdpi:critical}
In this section we prove the critical case of Theorem~\ref{thm:boht-non-recon-special}\ref{item:thm-boht-non-recon-special-small-d}, that for $\lambda\in \left[\frac 15,1\right]$ and $(r-1)d\lambda^2=1$, reconstruction is impossible.

The proof idea is similar to the SKL case (Section~\ref{sec:skl-multi-sdpi:critical}). Because $\chi^2$-capacity is a bounded function for BMS channels, the proof is easier than the SKL case.
\begin{proof}[Proof of Theorem~\ref{thm:boht-non-recon-special}\ref{item:thm-boht-non-recon-special-small-d} critical case]
Suppose for the sake of contraction that reconstruction holds.
Then the limit channel $M_\infty := \lim_{k\to \infty} M_k$ is non-trivial.
Let $\epsilon := C_{\chi^2}(M_\infty) >0$.
Then $M_\infty \le_{\ln} \BEC_{1-\epsilon}$, and thus $\BP(M_\infty) \le_{\ln} \BP(\BEC_{1-\epsilon})$. So
\begin{align}
  & C_{\chi^2}(\BP(M_\infty)) \le C_{\chi^2}(\BP(\BEC_{1-\epsilon}))
  \le d C_{\chi^2}\left(\BEC_{1-\epsilon}^{\times (r-1)}\circ B_{r,\lambda}\right) \\
  & \le (r-1) d \epsilon f_{r,\lambda}(\epsilon) < (r-1) d \lambda^2 \epsilon
  = \epsilon = C_{\chi^2}(M_\infty),
\end{align}
where the first step is by \cite[Theorem 1]{makur2018comparison}, the second step is by subadditivity, the third step is by $f_{r,\lambda}=f_{B_{r,\lambda}}$, the fourth step is by Lemma~\ref{lem:chi2-multi-sdpi-large-lambda}, the fifth step is by $(r-1)d\lambda^2=1$, and the sixth step is by definition of $\epsilon$.
On the other hand, $\BP(M_\infty)=M_\infty$, so $C_{\chi^2}(\BP(M_\infty)) = C_{\chi^2}(M_\infty)$. Contradiction.
\end{proof}

\section{Weak recovery threshold for HSBM} \label{sec:hsbm}
In this section we prove Theorem~\ref{thm:hsbm-special}.
Our proof uses a reduction from HSBM to BOHT, which works in a very general setting.
Let us define the general HSBM.
\begin{definition}[General HSBM \cite{angelini2015spectral,stephan2022sparse}] \label{defn:hsbm}
  Let $n\ge 1$ (number of vertices), $q\ge 2$ (number of communities), $r\ge 2$ (hyperedge size) be integers. Let $\pi \in \cP([q])$ be a distribution with full support.
  Let $\bfA \in \left(\bR_{\ge 0}^q\right)^{\otimes r}$ be a tensor satisfying
  \begin{align} \label{eqn:defn-hsbm-sym-tensor}
    a_{i_1,\ldots,i_r} = a_{i_{\sigma(1)},\ldots,i_{\sigma(r)}}
  \end{align}
  for any $i_1,\ldots, i_r\in [q]$, $\sigma\in \Aut([r])$.
  The hypergraph stochastic block model $\HSBM(n,q,r,\pi,\bfA)$ is defined as follows:
  Let $V = [n]$ be the set of vertices.
  Generate a random label $X_u$ for all vertices $u\in V$ i.i.d.~$\sim \pi$.
  Then for every $S=\{u_1,\ldots,u_r\}\in \binom{V}r$, add hyperedge $S$ to the hypergraph with probability $\frac{a_{X_{u_1},\ldots,X_{u_r}}}{\binom n{r-1}}$.
  The resulting pair $(X, G=(V,E))$ is the output of the model.
\end{definition}
Clearly the above definition generalizes the model $\HSBM(n,2,r,a,b)$ defined in the introduction.

Let $(X,G) \sim \HSBM(n,q,r,\pi,\bfA)$. We say the model admits weak recovery if there exists an estimator outputting a subset $S\subseteq V$ such that for some $\epsilon>0$, with probability $1-o(1)$, there exists $i,j\in [q]$ such that
\begin{align} \label{eqn:defn-hsbm-rec-prob-weak}
  \frac{\#\{v\in S: X_v=i\} }{\#\{v\in V: X_v=i\}} - \frac{\#\{v\in S: X_v=j\} }{\#\{v\in V: X_v=j\}} \ge \epsilon.
\end{align}
For $\HSBM(n,2,r,a,b)$, this definition agrees with the one we gave in the introduction.

In the HSBM, the expected degree (number of hyperedges containing a vertex) of a vertex with label $i\in [q]$ is $d_i \pm o(1)$, where
\begin{align} \label{eqn:deg-i}
  d_i = \sum_{i_1,\ldots,i_{r-1}\in [q]} a_{i,i_1,\ldots,i_{r-1}} \prod_{j\in [r-1]} \pi_{i_j}.
\end{align}
If $d_i\ne d_j$ for some $i,j\in [q]$, we can distinguish community $i$ and $j$ using a classifier based on degree, which trivially solves the weak recovery problem.
Therefore, we make the following standard assumption.
\begin{condition} \label{cond:hsbm-deg-indis}
  We say the model $\HSBM(n,q,r,\pi,\bfA)$ is degree indistinguishable if $d_i=d_j$ for all $i,j\in [q]$, where $d_i$ is defined in Eq.~\eqref{eqn:deg-i}.
  For such models, we define $d=d_i$ for any $i$.
\end{condition}

For a degree indistinguishable HSBM, the local neighborhood of any vertex corresponds to a BOHT model.
This relationship was first shown in \cite{massoulie2014community,mossel2015reconstruction} in the case of two-community symmetric SBMs, and later generalized to various settings \cite{bordenave2015non,caltagirone2017recovering,stephan2019robustness,stephan2022non,gu2020non,chin2020optimal,chin2021optimal,mossel2022exact,pal2021community,stephan2022sparse}.
\begin{proposition}[HSBM-BOHT coupling {\cite[Prop.~3]{stephan2022sparse}}] \label{prop:hsbm-boht-coupling}
  Let $(X,G) \sim \HSBM(n,q,r,\pi,\bfA)$ be a model satisfying Condition~\ref{cond:hsbm-deg-indis}.
  Let $v\in V$ and $k=c \log n$ for some small enough constant $c>0$ not depending on $n$. Let $B(v,k)$ be the set of vertices with distance $\le k$ to $v$.

  Let $(T,\sigma) \sim \BOHT(q,r,\pi,M,\Pois(d))$, and
  \begin{align}
    M_{i,(i_1,\ldots,i_{r-1})} = \frac 1d a_{i,i_1,\ldots,i_{r-1}} \prod_{j\in [r-1]} \pi_j.
  \end{align}
  Let $\rho$ be the root of $T$, and $T_k$ be the set of vertices at distance $\le k$ to $\rho$.

  Then $(G|_{B(v,k)}, X_{B(v,k)})$ can be coupled to $(T_k, \sigma_{T_k})$ with $o(1)$ TV distance.
\end{proposition}
In the setting of Prop.~\ref{prop:hsbm-boht-coupling}, we say the model $\BOHT(q,r,\pi,M,\Pois(d))$ is the BOHT model corresponding to $\HSBM(n,q,r,\pi,\bfA)$.

Now we can state the general reduction. This reduction was first established by \cite{mossel2015reconstruction} in the case of two-community symmetric SBMs, and later generalized to various settings \cite{gu2020non,mossel2022exact}.
\begin{theorem}[{\cite[Theorem 5.15]{gu2023channel}}]  \label{thm:hsbm-reduction}
  Let $\HSBM(n,q,r,\pi,\bfA)$ be a model satisfying Condition~\ref{cond:hsbm-deg-indis}.
  Let $\BOHT(q,r,\pi,M,\Pois(d))$ be the corresponding BOHT model.
  If reconstruction for the BOHT model is impossible, then weak recovery for the HSBM is impossible.
\end{theorem}
The proof of Theorem~\ref{thm:hsbm-reduction} uses Prop.~\ref{prop:hsbm-boht-coupling} and that HSBMs have no long range correlations, a fact first established by \cite{mossel2015reconstruction} in the case of two-community symmetric SBMs.
\begin{proposition}[{\cite[Prop.~5.6]{gu2023channel}}] \label{prop:hsbm-approx-cond-indep}
  Let $(X,G=(V,E))\sim \HSBM(n,q,r,\pi,\bfA)$.
  Let $A=A(G),B=B(G),C=C(G) \subseteq V$ be a (random) partition of $V$ such that $B$ separates $A$ and $C$ in $G$ (i.e., there exists no hyperedges $S\in E$ intersecting both $A$ and $C$).
  If $|A\cup B| = o(\sqrt n)$ a.a.s., then
  \begin{align}
    \bP(X_A | X_{B\cup C}, G) = (1\pm o(1)) \bP(X_A | X_B, G)~\aas
  \end{align}
\end{proposition}
We omit the proofs of Theorem~\ref{thm:hsbm-reduction} and Prop.~\ref{prop:hsbm-approx-cond-indep} and refer the reader to \cite{gu2023channel},

Using Theorem~\ref{thm:hsbm-reduction} and multi-terminal contraction coefficients we can prove impossibility of weak recovery results for HSBMs.
\begin{corollary}[Impossibility of weak recovery for general HSBM] \label{coro:hsbm-general}
  Let $\HSBM(n,q,r,\pi,\bfA)$ be a model satisfying Condition~\ref{cond:hsbm-deg-indis}.
  Let $\BOHT(q,r,\pi,M,\Pois(d))$ be the corresponding BOHT model.
  If any of the conditions in Theorem~\ref{thm:boht-non-recon}\ref{item:thm-boht-non-recon:general}\ref{item:thm-boht-non-recon:bms} holds, then weak recovery is impossible.
\end{corollary}
\begin{proof}
  By Theorem~\ref{thm:hsbm-reduction} and Theorem~\ref{thm:boht-non-recon}.
\end{proof}

In particular, for $\HSBM(n,2,r,a,b)$, we get Theorem~\ref{thm:hsbm-special}.
\begin{proof}[Proof of Theorem~\ref{thm:hsbm-special}]
  By Theorem~\ref{thm:hsbm-reduction} and Theorem~\ref{thm:boht-non-recon-special}.
\end{proof}

\section{Reconstruction below the KS threshold} \label{sec:recon-large-deg}
In this section we prove Theorem~\ref{thm:boht-recon-large-d}, that for the special BOHT model on a regular or Poisson hypertree, reconstruction is possible below the KS threshold for $r\ge 7$ and $d$ large enough.
Our proof is an analysis of evolution of $\chi^2$-capacity (also called magnetization in literature) and Gaussian approximation for large degree.

\subsection{Behavior of \texorpdfstring{$\chi^2$}{chi2}-capacity} \label{sec:recon-large-deg:large-deg-step}

\begin{proposition}[Large degree asymptotics] \label{prop:large-deg-step}
  Fix $r\in \bZ_{\ge 2}$. For any $\epsilon>0$, there exists $d_0=d_0(r,\epsilon)>0$ such that for any $d\ge d_0$ and $\lambda\in [0,1]$ with $(r-1)d \lambda^2 \le 1$, for any BMS channel $P$ we have
  \begin{align}
    | C_{\chi^2}(\BP(P)) - g_{r,d,\lambda} (C_{\chi^2}(P)) | \le \epsilon,
  \end{align}
  where
  \begin{align}
    g_{r,d,\lambda}(x) &:= \bE_{Z\sim \cN(0,1)} \tanh\left(s_{r,d,\lambda}(x) + \sqrt{s_{r,d,\lambda}(x)} Z\right), \\
    s_{r,d,\lambda}(x) &:= d\lambda^2 \cdot \frac 12 \left((1+x)^{r-1} - (1-x)^{r-1}\right).
  \end{align}
\end{proposition}
The rest of this section is devoted to the proof of Prop.~\ref{prop:large-deg-step}.

We first describe $\BP(P)$ in terms of the $\theta$-component.
Let $P$ be a BMS channel and $P_\theta$ be the $\theta$-component of $P$.
Let $t$ be the offspring ($t=d$ for regular hypertrees, $t\sim \Pois(d)$ for Poisson hypertrees).
Let $(\theta_{ij})_{i\in [t], j\in [r-1]}$ generated $\iidsim P_\theta$, where $\theta_{ij}$ is the $\theta$-component of the $j$-th vertex in the $i$-th downward hyperedge.
Let $\theta_i$ be the $\theta$-component of $i$-th hyperedge $P^{\times (r-1)}\circ B$.
As discussed in Section~\ref{sec:non-recon-special}, given $(\theta_{ij})_{j\in [r-1]}$, $\theta_i$ is equal to (the absolute value of)
\begin{align}
  \frac{\lambda \left( \prod_{j\in [r-1]} \left(1+\theta_{ij} x_{ij}\right) - \prod_{j\in [r-1]} \left(1-\theta_{ij} x_{ij}\right)\right)}{\lambda \left(\prod_{j\in [r-1]} \left(1+\theta_{ij} x_{ij}\right) +\prod_{j\in [r-1]} \left(1-\theta_{ij} x_{ij}\right)\right) + 2(1-\lambda)}.
\end{align}
with probability
\begin{align}
  \lambda \left(\prod_{j\in [r-1]} \left(\frac 12+\frac 12\theta_{ij} x_{ij}\right) +\prod_{j\in [r-1]} \left(\frac 12-\frac 12\theta_{ij} x_{ij}\right)\right) + 2^{2-r}(1-\lambda)
\end{align}
for $(x_{ij})_{j\in [r-1]} \in \{\pm\}^{r-1}$, $x_{i1}=+$.

Let $\ol\theta$ be the $\theta$-component of the full channel $\BP(P)$.
Let $P_{\ol\theta}$ denote the distribution of $\ol\theta$.
Then given $(\theta_i)_{i\in [t]}$,
$\ol\theta$ is equal to (the absolute value of)
\begin{align}
  \frac{\prod_{i\in [t]}(1+\theta_i x_i) - \prod_{i\in [t]}(1-\theta_i x_i)}{\prod_{i\in [t]}(1+\theta_i x_i) + \prod_{i\in [t]}(1-\theta_i x_i)}
\end{align}
with probability
\begin{align}
  \prod_{i\in [t]}\left(\frac 12+\frac 12\theta_i x_i\right) + \prod_{i\in [t]}\left(\frac 12-\frac 12\theta_i x_i\right)
\end{align}
for $(x_1,\ldots,x_t)\in \{\pm\}^t$, $x_1=+$.
In other words,
\begin{align}
  &~P_{\ol\theta | \theta_1,\ldots,\theta_i} \\
  \nonumber =&~ \sum_{(x_1,\ldots,x_t)\in \{\pm\}^t} \left(\prod_{i\in [t]} \left(\frac 12 +\frac 12 \theta_i x_i\right)\right)
  \mathbbm{1}\left\{\left|\frac{\prod_{i\in [t]}(1+\theta_i x_i) - \prod_{i\in [t]}(1-\theta_i x_i)}{\prod_{i\in [t]}(1+\theta_i x_i) + \prod_{i\in [t]}(1-\theta_i x_i)}\right|\right\} \\
  \nonumber =&~ \sum_{(x_1,\ldots,x_t)\in \{\pm\}^t} \left(\prod_{i\in [t]} \left(\frac 12 +\frac 12 \theta_i x_i\right)\right)
  \mathbbm{1}\left\{\left| \tanh\left(\sum_{i\in [t]} \arctanh(\theta_i x_i) \right) \right|\right\}.
\end{align}
Write $\wt \theta_i=\theta_i x_i$. Then $\bP[\wt \theta_i = s \theta_i | \theta_i] = \frac 12 + \frac 12 \theta_i s$ for $s\in \{\pm\}$.
So $\wt \theta_i$ for $i\in [t]$ are iid generated from the same distribution.
Let us call this distribution $D$.
Then
\begin{align}
  P_{\ol\theta} = \bE_t \bE_{\wt\theta_1,\ldots,\wt\theta_t\iidsim D} \mathbbm{1}\left\{\left| \tanh\left(\sum_{i\in [t]} \arctanh \wt\theta_i \right) \right|\right\}
\end{align}
This allows us to use central limit theorems to control the behavior of $\sum_{i\in [t]} \arctanh\wt\theta_i$.
\begin{lemma} \label{lem:proof-large-deg-step:per-hyperedge}
  There exists a constant $d_0=d_0(r)>0$ such that for any $d>d_0$, $\lambda\in[0,1]$ with $(r-1)d\lambda^2\le 1$, and any BMS channel $P$, we have
  \begin{align}
    & \left| C_{\chi^2}(P^{\times (r-1)} \circ B) - s_{r,\lambda}(C_{\chi^2}(P)) \right| \le O_r(\lambda^3),\\
    & s_{r,\lambda} (x) := \lambda^2 \cdot \frac 12 \left((1+x)^{r-1}-(1-x)^{r-1}\right).
  \end{align}
  where $O_r$ hides a multiplicative factor depending only on $r$.
\end{lemma}
\begin{proof}
  We have
  \begin{align*}
    &~ C_{\chi^2}(P^{\times (r-1)} \circ B)\\
    =&~ \bE \theta_i^2\\
    =&~ \underset{\theta_{i1},\ldots,\theta_{i,r-1}\iidsim P_\theta}{\bE} \sum_{\substack{(x_{ij})_{j\in [r-1]}\in \{\pm\}^{r-1}\\ x_{i1}=+}} \\
    &~\qquad \left(2^{1-r} \cdot \frac{\lambda^2 \left( \prod_{j\in [r-1]} \left(1+\theta_{ij} x_{ij}\right) - \prod_{j\in [r-1]} \left(1-\theta_{ij} x_{ij}\right)\right)^2}{\lambda \left(\prod_{j\in [r-1]} \left(1+\theta_{ij} x_{ij}\right) +\prod_{j\in [r-1]} \left(1-\theta_{ij} x_{ij}\right)\right)+ 2(1-\lambda)}\right)\\
    =&~ \underset{\theta_{i1},\ldots,\theta_{i,r-1}\iidsim P_\theta}{\bE} \sum_{\substack{(x_{ij})_{j\in [r-1]}\in \{\pm\}^{r-1}\\ x_{i1}=+}} \\
    &~\qquad \left(2^{-r} \lambda^2 \left( \prod_{j\in [r-1]} \left(1+\theta_{ij} x_{ij}\right) - \prod_{j\in [r-1]} \left(1-\theta_{ij} x_{ij}\right)\right)^2\right) + O_r(\lambda^3).
  \end{align*}
  The inner summation satisfies
  \begin{align*}
    &~ \sum_{\substack{(x_{ij})_{j\in [r-1]}\in \{\pm\}^{r-1}\\ x_{i1}=+}}
    \left( \prod_{j\in [r-1]} \left(1+\theta_{ij} x_{ij}\right) - \prod_{j\in [r-1]} \left(1-\theta_{ij} x_{ij}\right)\right)^2 \\
    =&~ \frac 12 \sum_{\substack{(x_{ij})_{j\in [r-1]}\in \{\pm\}^{r-1}}}
    \left( \prod_{j\in [r-1]} \left(1+\theta_{ij} x_{ij}\right) - \prod_{j\in [r-1]} \left(1-\theta_{ij} x_{ij}\right)\right)^2 \\
    =&~ \frac 12 \sum_{\substack{(x_{ij})_{j\in [r-1]}\in \{\pm\}^{r-1}}}
    \left( \prod_{j\in [r-1]} \left(1+2 \theta_{ij} x_{ij}+\theta_{ij}^2\right)
    - 2 \prod_{j\in [r-1]} \left(1-\theta_{ij}^2\right) \right. \\
    & \qquad \left.+ \prod_{j\in [r-1]} \left(1-2 \theta_{ij} x_{ij}+\theta_{ij}^2\right)\right)\\
    =&~ 2^{r-1} \left(\prod_{j\in [r-1]} (1+\theta_{ij}^2)-\prod_{j\in [r-1]} (1-\theta_{ij}^2)\right).
  \end{align*}
  Therefore
  \begin{align*}
    &~ \underset{\theta_{i1},\ldots,\theta_{i,r-1}\iidsim P_\theta}{\bE} \sum_{\substack{(x_{ij})_{j\in [r-1]}\in \{\pm\}^{r-1}\\ x_{i1}=+}} \left( \prod_{j\in [r-1]} \left(1+\theta_{ij} x_{ij}\right) - \prod_{j\in [r-1]} \left(1-\theta_{ij} x_{ij}\right)\right)^2\\
    =&~ 2^{r-1} \underset{\theta_{i1},\ldots,\theta_{i,r-1}\iidsim P_\theta}{\bE} \left(\prod_{j\in [r-1]} (1+\theta_{ij}^2)-\prod_{j\in [r-1]} (1-\theta_{ij}^2)\right) \\
    =&~ 2^{r-1} \left(\left(1+C_{\chi^2}(P)\right)^{r-1}-\left(1-C_{\chi^2}(P)\right)^{r-1}\right).
  \end{align*}
  Combining everything we finish the proof.
\end{proof}

\begin{lemma} \label{lem:proof-large-deg-step:exp-and-var}
  There exists a constant $d_0=d_0(r)>0$ such that for any $d>d_0$, $\lambda\in[0,1]$ with $(r-1)d\lambda^2\le 1$, and any BMS channel $P$, we have
  \begin{align}
    \left|\bE \arctanh \wt \theta_i - s_{r,\lambda}(C_{\chi^2}(P))\right| &= O_r(\lambda^3),\\
    \left|\Var(\arctanh \wt \theta_i) - s_{r,\lambda}(C_{\chi^2}(P))\right| &= O_r(\lambda^3).
  \end{align}
\end{lemma}

\begin{proof}
  Note that $\theta_i = O_r(\lambda)$ almost surely.
  When $d$ is large enough, $\lambda$ is small enough, and $\arctanh \theta_i = \theta_i + O_r(\lambda^3)$ almost surely by Taylor expansion. Then
  \begin{align}
    \bE \arctanh \wt \theta_i = \bE [\theta_i \arctanh \theta_i] = \bE \theta_i^2 + O_r(\lambda^4),\\
    \bE (\arctanh \wt \theta_i)^2 = \bE (\arctanh \theta_i)^2 = \bE \theta_i^2 + O_r(\lambda^4).
  \end{align}
  By Lemma~\ref{lem:proof-large-deg-step:per-hyperedge}, we have
  \begin{align}
    &\bE \theta_i^2 = s_{r,\lambda}(C_{\chi^2}(P)) + O_r(\lambda^3).
  \end{align}
  This already implies the statement on $\bE \arctanh \wt \theta_i$.
  For the statement on $\Var(\arctanh \wt \theta_i)$, we note that
  \begin{align}
    \bE \arctanh \wt \theta_i = s_{r,\lambda}(C_{\chi^2}(P)) + O_r(\lambda^3) = O_r(\lambda^2).
  \end{align}
  So
  \begin{align}
    \Var(\arctanh \wt \theta_i) &= \bE (\arctanh \wt \theta_i)^2 - \left(\bE \arctanh \wt \theta_i\right)^2\\
    \nonumber &= s_{r,\lambda}(C_{\chi^2}(P)) + O_r(\lambda^3).
  \end{align}
  This finishes the proof.
\end{proof}

Now we recall a normal approximation result from \cite[Prop.~5.3]{mossel2022exact}. We only need the scalar version of it.
\begin{lemma}[\cite{mossel2022exact}] \label{lem:normal-approx-lindeberg}
  Let $\phi: \bR \to \bR$ be a thrice differentiable and bounded function with bounded derivatives up to third order. Let $V_1,\ldots, V_t\in \bR$ be independent random real numbers.
  Suppose there exists deterministic numbers $\mu,\sigma\in \bR$ such that the following holds:
  for some constant $C>0$, almost surely
  \begin{align}
    &\max\left\{\left| \sum_{j\in [t]} \bE V_j -\mu \right|, \left|\sum_{j\in [t]} \Var(V_j) - \sigma^2\right| \right\}\le C t^{-1/2}, \\
    &\max\left\{ |\mu|, |\sigma^2| \right\} \le C, \qquad \max_{j\in [t]} |V_j| \le C t^{-1/2}.
  \end{align}
  Then for any $\epsilon>0$, there exists $t_0 = t_0(\epsilon,\phi,C)$ such that if $t>t_0$, then
  \begin{align}
    \left| \bE \phi \left( \sum_{j\in [t]} V_j \right) - \bE_{W\sim \cN(\mu,\sigma^2)} \phi(W) \right| \le \epsilon.
  \end{align}
\end{lemma}

We now have everything we need for the proof of Prop.~\ref{prop:large-deg-step}.
\begin{proof}[Proof of Prop.~\ref{prop:large-deg-step}]
  \textbf{Regular hypertree:}
  Define $\wt \theta$ as $\bP[\wt \theta = s \theta | \theta] = \frac 12 + \theta s$ for $s\in \{\pm\}$. Then
  \begin{align}
    C_{\chi^2}(\BP(P)) = \bE \wt \theta = \bE_{\wt\theta_1,\ldots,\wt\theta_t\iidsim D} \tanh\left(\sum_{i\in [t]} \arctanh \wt\theta_i \right) .
  \end{align}
  In fact, the equality holds with $\tanh$ replaced by $\tanh^2$. We use the $\tanh$ form here because it is slightly simpler.

  Now we apply Lemma~\ref{lem:normal-approx-lindeberg} with
  \begin{align}
    \phi(x) = \tanh x,\quad V_i = \arctanh \wt \theta_i, \quad \mu = \sigma^2 = d s_{r,\lambda}(C_{\chi^2}(P)) = s_{r,d,\lambda}(C_{\chi^2}(P)).
  \end{align}
  The conditions in Lemma~\ref{lem:normal-approx-lindeberg} are satisfied by Lemma~\ref{lem:proof-large-deg-step:exp-and-var} and because $\lambda = O(d^{-1/2})$.
  This finishes the proof.

  \textbf{Poisson hypertree:}
  Fix $\epsilon>0$. Let $t\sim \Pois(d)$.
  By Poisson tail bounds, we have $\bP[|t-d| > d^{0.6}] < \epsilon/3$ for large enough $d$ (depending only on $\epsilon$).
  We apply Lemma~\ref{lem:normal-approx-lindeberg} for every $t\in [d-d^{0.6},d+d^{0.6}]$, with $\mu = \sigma^2 = s_{r,t,\lambda} (C_{\chi^2}(P))$ and error tolerance $\epsilon/3$.
  Note that
  \begin{align}
    \left|s_{r,d,\lambda} (C_{\chi^2}(P)) - s_{r,t,\lambda} (C_{\chi^2}(P))\right| = O_r(d^{-0.4}).
  \end{align}
  So for $d$ large enough (depending only on $\epsilon,r$), we have
  \begin{align}
    \left|g_{r,d,\lambda} (C_{\chi^2}(P)) - g_{r,t,\lambda} (C_{\chi^2}(P))\right| \le \epsilon/3
  \end{align}
  by continuity of $g_r$ (Lemma~\ref{lem:large-deg-step-g-monotone}).

  Therefore we have
  \begin{align*}
    &~\left|C_{\chi^2}(\BP(P)) - g_{r,d,\lambda}(C_{\chi^2}(P))\right| \\
    =&~\left|\bE_{t\sim \Pois(d)} C_{\chi^2}((P^{\times (r-1)}\circ B)^{\star t}) - g_{r,d,\lambda}(C_{\chi^2}(P))\right| \\
    \le&~ \bE_{t\sim \Pois(d)} \mathbbm{1}\{|t-d|\le d^{0.6}\}\left|C_{\chi^2}((P^{\times (r-1)}\circ B)^{\star t}) - g_{r,t,\lambda}(C_{\chi^2}(P))\right|\\
    &~+\bE_{t\sim \Pois(d)} \mathbbm{1}\{|t-d|\le d^{0.6}\} \left|g_{r,t,\lambda}(C_{\chi^2}(P))-g_{r,d,\lambda}(C_{\chi^2}(P))\right|\\
    &~+\bE_{t\sim \Pois(d)} \mathbbm{1}\{|t-d|> d^{0.6}\} |C_{\chi^2}((P^{\times (r-1)}\circ B)^{\star t}) - g_{r,d,\lambda}(C_{\chi^2}(P))|\\
    \le&~ \epsilon/3 + \epsilon/3 + \epsilon/3 = \epsilon.
  \end{align*}
  Note that $C_{\chi^2}(P)\in [0,1]$ for any BMS channel $P$, and $g_{r,d,\lambda}(x)\in [0,1]$ for all $x\in [0,1]$.
\end{proof}


\subsection{Properties of functions} \label{sec:recon-large-deg:func-property}
Theorem~\ref{thm:boht-recon-large-d} then follows from analyzing properties of the function $g_{r,d,\lambda}$.
For $r\ge 2$, we define
\begin{align}
  g_r(x) &:= \bE_{Z\sim \cN(0,1)} \tanh\left(s_r(x) + \sqrt{s_r(x)} Z\right),  \label{eqn:large-deg-step-g} \\
  s_r(x) &:= \frac 1{2(r-1)} \left((1+x)^{r-1} - (1-x)^{r-1}\right).
\end{align}

\begin{lemma} \label{lem:large-deg-step-g-monotone}
  For any $r\ge 2$, the function $g_r$ is strictly increasing and continuous differentiable on $[0,1]$.
\end{lemma}
\begin{proof}
  Note that $s_r(x)$ is continuous and increasing on $[0,1]$.
  Therefore it suffices to prove that
  \begin{align}
    g(s) := \bE_{Z\sim \cN(0,1)} \tanh \left(s + \sqrt s Z \right)
  \end{align}
  is continuous and increasing on $\bR_{\ge 0}$.
  This statement is in fact equivalent to the $q=2$ case in \cite[Lemma 4.4]{sly2011reconstruction}, after a suitable change of variables.
\end{proof}



\begin{lemma} \label{lem:large-deg-step-g-r7}
  For $r\ge 7$, there exists $x\in (0,1)$ such that $g_r(x)>x$.
\end{lemma}
\begin{proof}
    We can numerically verify that $g_7(0.8) > 0.8$.
    Note that $s_r(0.8)$ is increasing for $r\ge 7$.
    Therefore for $r\ge 7$, we have $g_r(0.8) \ge g_7(0.8) > 0.8$.
\end{proof}


\subsection{Proof of Theorem~\ref{thm:boht-recon-large-d}}
We are now ready to prove Theorem~\ref{thm:boht-recon-large-d}.
\begin{proof}[Proof of Theorem~\ref{thm:boht-recon-large-d}]
  Choose $x\in (0,1)$ so that $g_r(x) > x$ via Lemma~\ref{lem:large-deg-step-g-r7}.
  By continuity of $g_r$ (Lemma~\ref{lem:large-deg-step-g-monotone}), there exists $\epsilon>0$ such that $g_{r,d,\lambda}(x) > x + \epsilon$ for $(r-1)d\lambda^2=1-\epsilon$.
  Note that $g_{r,d,\lambda}(x)$'s dependence on $d$ and $\lambda$ is only through $d\lambda^2$.

  Take $d_0 = d_0(r,\epsilon)$ in Prop.~\ref{prop:large-deg-step}.
  For any $d>d_0$, choose $\lambda\in[0,1]$ such that $(r-1)d\lambda^2 = 1-\epsilon$.
  By Prop.~\ref{prop:large-deg-step}, choice of $\epsilon$, and Lemma~\ref{lem:large-deg-step-g-monotone}, for all BMS $P$ with $C_{\chi^2}(P) \ge x$ we have
  \begin{align}
    C_{\chi^2}(\BP(P)) \ge g_{r,d,\lambda}(C_{\chi^2}(P)) - \epsilon \ge x.
  \end{align}
  Therefore
  \begin{align}
    \lim_{k\to \infty} I_{\chi^2}(\sigma_\rho; T_k, \sigma_{L_k}) = \lim_{k\to \infty} C_{\chi^2}(M_k) \ge x.
  \end{align}
  Finally
  \begin{align}
    \lim_{k\to \infty} I(\sigma_\rho; T_k, \sigma_{L_k}) \ge \lim_{k\to \infty} \frac {\log e} 2 I_{\chi^2}(\sigma_\rho; T_k, \sigma_{L_k}) \ge \frac {x \log e}2,
  \end{align}
  where the first step is because $C(P) \ge \frac {\log e}2 C_{\chi^2}(P)$ for any BMS $P$.
\end{proof}

\end{document}